%% file: maximal.tex
\documentclass[a4paper]{amsart}

\usepackage{amssymb,pstricks,diagrams, amscd}
\usepackage[totalwidth=17cm,totalheight=24cm]{geometry}
\usepackage{graphicx}

\begin{document}

\input{macros}

\title{Maximal surfaces and the universal Teichm\"uller space}
\author{Francesco Bonsante}
\address{Dipartimento di Matematica\\
Universit\`a degli Studi di Pavia\\
via Ferrata 1\\
27100 Pavia, Italy}
\email{francesco.bonsante@unipv.it}
\author{Jean-Marc Schlenker}\thanks{Partially supported by the ANR programs 
FOG (ANR-07-BLAN-0251-01),
Repsurf (ANR-06-BLAN-0311)
and ETTT (ANR-09-BLAN-0116-01).}
\address{Institut de Math\'ematiques de Toulouse, UMR CNRS 5219\\
Universit\'e Toulouse III\\
31062 Toulouse cedex 9, France}
\email{schlenker@math.univ-toulouse.fr}
\date{November 2009 (v1)}

\begin{abstract}
We show that any element of the universal Teichm\"uller space is realized by a 
unique minimal Lagrangian diffeomorphism from the hyperbolic plane to itself. 
The proof uses maximal surfaces in the 3-dimensional anti-de Sitter space. We
show that, in $AdS^{n+1}$, any subset $E$ of the 
boundary at infinity which is the boundary at infinity of a space-like hypersurface bounds 
a maximal space-like hypersurface. In $AdS^3$, if $E$ is the graph of a quasi-symmetric
homeomorphism, then this maximal surface is unique, and it has negative sectional
curvature. As a by-product, we find a simple characterization
of quasi-symmetric homeomorphisms of the circle in terms of 3-dimensional 
projective geometry.
\end{abstract}
\maketitle


\input{jms1}

\input{fb1}

\input{jms2}

\appendix

\input{fb2}

\bibliographystyle{amsplain}
\bibliography{../outils/biblio}
\end{document}

%% file: macros.tex
\newtheorem{cor}{Corollary}[section]
\newtheorem{theorem}[cor]{Theorem}
\newtheorem{prop}[cor]{Proposition}
\newtheorem{lemma}[cor]{Lemma}
\newtheorem{conj}[cor]{Conjecture}
\newtheorem{claim}[cor]{Claim}
\theoremstyle{definition}
\newtheorem{defi}[cor]{Definition}
\theoremstyle{remark}
\newtheorem{remark}[cor]{Remark}
\newtheorem{example}[cor]{Example}

\newcommand{\cD}{{\mathcal D}}
\newcommand{\cL}{{\mathcal L}}
\newcommand{\cM}{{\mathcal M}}
\newcommand{\cT}{{\mathcal T}}
\newcommand{\cML}{{\mathcal M\mathcal L}}
\newcommand{\cGH}{{\mathcal G\mathcal H}}
\newcommand{\C}{{\mathbb C}}
\newcommand{\N}{{\mathbb N}}
\newcommand{\R}{{\mathbb R}}
\newcommand{\Z}{{\mathbb Z}}
\newcommand{\HH}{{\mathbb H}}
\newcommand{\Kt}{\tilde{K}}
\newcommand{\Mt}{\tilde{M}}
\newcommand{\dr}{{\partial}}
\newcommand{\kappab}{\overline{\kappa}}
\newcommand{\pib}{\overline{\pi}}
\newcommand{\Sigmab}{\overline{\Sigma}}
\newcommand{\gd}{\dot{g}}
\newcommand{\diff}{\mbox{Diff}}
\newcommand{\dev}{\mbox{dev}}
\newcommand{\tr}{\mbox{tr}}
\newcommand{\devb}{\overline{\mbox{dev}}}
\newcommand{\devt}{\tilde{\mbox{dev}}}
\newcommand{\vol}{\mbox{Vol}}
\newcommand{\hess}{\mbox{Hess}}
\newcommand{\db}{\overline{\partial}}
\newcommand{\Sigmat}{\tilde{\Sigma}}

\newcommand{\sh}{\mathrm{sinh}\,}
\newcommand{\ch}{\mathrm{cosh}\,}
\newcommand{\grad}{\mathrm{grad}\,}

\newcommand{\cunc}{{\mathcal C}^\infty_c}
\newcommand{\cun}{{\mathcal C}^\infty}
\newcommand{\dd}{d_D}
\newcommand{\dmin}{d_{\mathrm{min}}}
\newcommand{\dmax}{d_{\mathrm{max}}}
\newcommand{\Dom}{\mathrm{Dom}}
\newcommand{\dn}{d_\nabla}
\newcommand{\ded}{\delta_D}
\newcommand{\delmin}{\delta_{\mathrm{min}}}
\newcommand{\delmax}{\delta_{\mathrm{max}}}
\newcommand{\hmin}{H_{\mathrm{min}}}
\newcommand{\maxi}{\mathrm{max}}
\newcommand{\oL}{\overline{L}}
\newcommand{\oP}{{\overline{P}}}
\newcommand{\Ran}{\mathrm{Ran}}
\newcommand{\tgamma}{\tilde{\gamma}}
\newcommand{\cotan}{\mbox{cotan}}
\newcommand{\lambdat}{\tilde\lambda}
\newcommand{\St}{\tilde S}

\newcommand{\II}{I\hspace{-0.1cm}I}
\newcommand{\III}{I\hspace{-0.1cm}I\hspace{-0.1cm}I}
\newcommand{\note}[1]{\marginpar{\tiny #1}}

%% file: jms1.tex
\section{Introduction}

\subsection{The universal Teichm\"uller space}

We consider here the universal Teichm\"uller space $\cT$, which can be
defined as the space of quasi-symmetric homeomorphisms from $S^1$ to
$S^1$ up to projective transformations, see
e.g. \cite{gardiner-harvey}. The quasi-symmetric homeomorphisms of
$S^1$ to $S^1$ are precisely the homeomorphisms which are the boundary
value of a quasi-conformal diffeomorphism from $\mathbb H^2$ to
$\mathbb H^2$, so that the universal Teichm\"uller space $\cT$ can be
defined as the space of quasi-conformal diffeomorphisms from $\mathbb H^2$ to
$\mathbb  H^2$, up to composition with a hyperbolic isometry and up to the
equivalence relation which identifies two quasi-conformal
diffeomorphisms if they have the same boundary value.

It was conjectured by Schoen that any element in the universal Teichm\"uller
space can be uniquely realized as a quasi-conformal harmonic diffeomorphism:

\begin{conj}[Schoen \cite{schoen:role}] \label{cj:schoen}
Let $\phi:S^1\rightarrow S^1$ be a quasi-symmetric
homeomorphism. There is a unique quasi-conformal harmonic
diffeomorphism $\psi:\mathbb H^2\rightarrow \mathbb H^2$ such that
$\partial \psi=\phi$.
\end{conj}

A number of partial results were obtained towards this conjecture,
proving the uniqueness of $\psi$ and its existence if $\phi$ is smooth
enough, see \cite{tam-wan,akutagawa:harmonic,markovic:harmonic} and
the references there. 

\subsection{Minimal Lagrangian diffeomorphisms}

Our first goal here is to prove an analog of Conjecture
\ref{cj:schoen}, with harmonic maps replaced by close relatives: minimal
Lagrangian diffeomorphisms. 

\begin{defi}
Let $\Phi:S\rightarrow S'$ be a diffeomorphism between two hyperbolic
surfaces. $\Phi$ is {\em minimal Lagrangian} if it is area-preserving, and
its graph is a minimal surface in $S\times S'$. 
\end{defi}

The relationship between harmonic maps and minimal Lagrangian maps is
as follows.

\begin{prop}
\begin{itemize}
\item Let $S_0$ be a Riemann surface, and let $\psi:S_0\rightarrow S$ be a 
harmonic diffeomorphism from $S_0$ to a hyperbolic surface $S$. Let 
$q$ be the Hopf differential of $\psi$. There is a unique 
harmonic diffeomorphism $\psi':S_0\rightarrow S'$ from $S_0$ to 
another hyperbolic surface $S'$ with Hopf differential $-q$.
Then $\psi'\circ\psi^{-1}:S\rightarrow S'$ is a minimal Lagrangian map.
\item Conversely, let $\Phi:S\rightarrow S'$ be a minimal Lagrangian
map between two (oriented) hyperbolic surfaces, and let $S_0$ be the
graph of $\Phi$, considered as a Riemann surface with the complex
structure defined by its induced metric in $S\times S'$. Then the
natural projections from $S_0$ to $S$ and to $S'$ are harmonic maps,
and the sum of their Hopf differentials is zero.
\end{itemize}
\end{prop}

Thus minimal Lagrangian maps are a kind of ``symmetric squares'' of 
harmonic maps. 

It is known that any diffeomorphism between two closed hyperbolic
surfaces can be deformed to a unique harmonic diffeomorphism, see
e.g. \cite{jost,jost:compact}.  In the same manner, it was proved by
Schoen and by Labourie that any such diffeomorphism can be deformed to
a unique minimal Lagrangian diffeomorphism \cite{L5,schoen:role}.

Our first result is an extension of this existence and uniqueness
result to the universal Teichm\"uller space.

\begin{theorem} \label{tm:ml}
Let $\phi:S^1\rightarrow S^1$ be a quasi-symmetric
homeomorphism. There is a unique quasi-conformal minimal Lagrangian
diffeomorphism $\Phi:\mathbb H^2\rightarrow \mathbb H^2$ such that $\partial
\Phi=\phi$.
\end{theorem}

The result of Schoen and Labourie on closed hyperbolic surfaces
obviously follows from this.  The proof of Theorem \ref{tm:ml} can be
found in Section \ref{sc:proof}. Note that partial results in this
direction were obtained previously by Aiyama, Akutagawa and Wan \cite{AAW}, 
who proved the existence part of Theorem \ref{tm:ml} when $\phi$ has small
dilatation. Recently, Brendle has obtained results on the
existence and uniqueness of minimal Lagrangian diffeomorphisms between
two convex domains of the same, finite area in hyperbolic surfaces,
see \cite{brendle:minimal}.

\subsection{The anti-de Sitter space}

The proof of Theorem \ref{tm:ml} is essentially based on the geometry
of maximal surfaces in the anti-de Sitter (AdS) 3-dimensional
space, as already in \cite{AAW}. This follows a pattern in some recent works (see
\cite{mess,mess-notes,cone,mbh,earthquakes} and also \cite{AAW}),
where results on Teichm\"uller theory were proved using 3-dimensional
AdS geometry, although mostly in a somewhat different direction. 
The relationship between maximal surfaces in
3-dimensional AdS manifolds and minimal Lagrangian maps between closed
hyperbolic surfaces was also used recently in \cite{minsurf}.

The 3-dimensional AdS space can be considered as a Lorentzian analog of the 
3-dimensional hyperbolic space. It can be defined as the quadric
\[
  AdS^*_{3}=\{x\in\R^{2,2}|\langle x,x\rangle=-1\}~, 
\]
where $\R^{2,2}$ is $\R^{4}$ endowed with bilinear symmetric form of
signature $(2,2)$. It is a geodesically complete Lorentz manifold of
constant curvature $-1$. Another way to define it 
is as the Lie group $SL(2,\R)$, endowed
with its bi-invariant Killing metric.  More details are given in
Section \ref{sc:ads}. The key point for us, basically discovered by
Mess \cite{mess,mess-notes} and used in the references mentioned
above, is that space-like surfaces in $AdS^*_3$ naturally give rise
to area-preserving diffeomorphisms from the hyperbolic plane to
itself. In this way, Theorem \ref{tm:ml} is proved below to be
equivalent to an existence and uniqueness statement for maximal
space-like surfaces in $AdS^*_3$, and it is in this form that it is
proved.

The anti-de Sitter space can of course be defined in higher
dimensions. The existence part of the result on maximal surfaces is
actually stated (and proved) below in the more general context of
maximal hypersurfaces in $AdS^*_{n+1}$, see Theorem
\ref{tm:existence}.  The uniqueness part, however, is considered here
only in $AdS^*_3$ (and it needs hypothesis that are more interesting
in dimension $2+1$), see Theorem \ref{tm:uniqueness}.

\subsection{Maximal surfaces and minimal Lagrangian diffeomorphisms}

For closed hyperbolic surfaces, the existence of a minimal Lagrangian
diffeomorphism is equivalent to the existence of a maximal space-like
surface in a 3-dimensional globally hyperbolic AdS manifold, see
\cite{minsurf}. This relation extends to maximal surfaces in 
$AdS^*_3$ and the universal Teichm\"uller space as follows.

One way to consider the bridge between Teichm\"uller theory and AdS geometry 
is through the asymptotic boundary of $AdS_3^*$ -- denoted by $\partial_\infty
AdS_3^*$ -- that, as for the hyperbolic space, furnishes a natural
compactification of $AdS_3^*$. As in the hyperbolic case, a conformal
Lorentzian structure is defined on $\partial_\infty AdS_3^*$.  There
is a natural projection $\partial_\infty AdS_3^*$ to $S^1\times S^1$
that is a $2$-to-$1$ covering (see Section \ref{ssc:3d} for details).
The graph of any homeomorphism of $S^1$ lifts to a spacelike closed
curve in $\partial_\infty AdS_3^*$.

\begin{prop} \label{pr:translation}
\begin{itemize}
\item Let $S\subset AdS^*_3$ be a maximal space-like graph with
  unformly negative sectional curvature. Then there is a minimal Lagrangian
  diffeomorphism $\Phi_S:\HH^2\rightarrow \HH^2$ associated to $S$, and
   the graph of $\partial \Phi_S:\dr \HH^2\rightarrow\dr \HH^2$ is the
   projection of the boundary at infinity of $S$ in $\dr_\infty AdS^*_3$.
\item Conversely, to any quasi-conformal minimal Lagrangian
  diffeomorphism $\Phi:\HH^2\rightarrow \HH^2$ is associated a maximal
  surface $S$ with uniformly negative sectional curvature and with
  boundary at infinity equal to the lifting of the graph of $\dr\Phi$
  in $\dr_\infty AdS_3^*$.
\end{itemize}
\end{prop}

It is this proposition which provides the bridge between Theorem \ref{tm:ml} 
and the existence and uniqueness of maximal surfaces in $AdS^*_3$.

\subsection{Existence and regularity of maximal hypersurfaces in $AdS_{n+1}$}

We can state an existence result for maximal hypersurfaces in the AdS space
with fixed boundary values. The regularity conditions on the boundary values
are quite weak, since we only demand that it bounds some space-like surface
in $AdS_{n+1}^*$. 

\begin{theorem} \label{tm:existence}
Let $\Gamma$ be a closed acausal $\mathrm C^{0,1}$ graph in $\dr_\infty AdS_{n+1}^*$
($n\geq 2$).
If $\Gamma$ does not contain lightlike segments, then there is a maximal
space-like hypersurface $S_0$ such that $\dr S_0=\Gamma$.
\end{theorem}

We provide in Section \ref{sc:maximal} a direct proof of this result, where the maximal
surface is obtained as a limit of bigger and bigger maximal disks.

This existence result can be improved insofar as the regularity of the 
hypersurface is concerned. To state this improvement, we need a definition. 
Let $\Gamma$ be a nowhere time-like 
graph in $\dr_\infty AdS_{n+1}^*$. Using the projective model
of $AdS_{n+1}^*$ which is also recalled in Section \ref{proj:sec}, we can consider
the convex hull of $\Gamma$, it is a convex subset of $AdS_{n+1}^*$ with
boundary at infinity containing $\Gamma$, we use the notation $CH(\Gamma)$. 
We denote by $C(\Gamma)$ the intersection with $AdS_{n+1}^*$ of $CH(\Gamma)$
(considered as a subset of projective space). The boundary of $C(\Gamma)$
is the disjoint union of two nowhere time-like hypersurfaces, which we call
$\dr_+C(\Gamma)$ and $\dr_-C(\Gamma)$. 

\begin{defi} \label{df:width}
The \emph{width} of $C(\Gamma)$ (or by extension of $\Gamma$),
denoted by $w(C(\Gamma))$ (resp. $w(\Gamma)$)
is the supremum of the (time) distance between $\dr_-C(\Gamma)$ and 
$\dr_+C(\Gamma)$.
\end{defi}

It is proved below (Lemma \ref{lm:width}) that $w(\Gamma)$ is always 
at most equal to $\pi/2$.

\begin{theorem} \label{tm:existence2}
Suppose that $w(\dr_\infty S)<\pi/2$ in Theorem \ref{tm:existence}. Then
$S_0$ can be taken to have bounded second fundamental form.
\end{theorem}

The proof is also in Section \ref{sc:proof}.

\subsection{The mean curvature flow}

We also give in the appendix another proof of Theorem \ref{tm:existence}. It
is based on the mean curvature flow for hypersurfaces in the anti-de Sitter
space. 

\begin{theorem} \label{tm:flow}
Let $S\subset AdS_{n+1}$ be a space-like graph. There exists a 
long-time solution of the mean curvature flow with initial value $S$ with
fixed boundary at infinity, defined for all $t>0$.
\end{theorem}

This flow converges, as $t\rightarrow \infty$, to a maximal surface. 
When $w(\dr_\infty S)<\pi/2$, we also have bounds on the second fundamental
form of the hypersurfaces occuring in the flow.

\subsection{Uniqueness of maximal surfaces in $AdS_3^*$}

We do not know whether maximal hypersurfaces with given boundary at 
infinity are unique in $AdS_{n+1}^*$. We can however state a result
for surfaces in $AdS_3*$, under a regularity assumption on the boundary
at infinity.

\begin{theorem} \label{tm:uniqueness}
Let $S$ be a space-like graph in $AdS_3^*$. Suppose that the boundary
at infinity of $S$ is the graph of a quasi-symmetric homeomorphism
from $S^1$ to $S^1$. Then there is a unique maximal surface in 
$AdS_3^*$ with boundary at infinity $\dr_\infty S$ and with bounded
second fundamental form, and it has negative sectional curvature.
\end{theorem}

The proof, which can be found in Section \ref{sc:proof}, 
is based on the following proposition.

\begin{prop}
Let $S_0\subset AdS_3$ be a maximal space-like graph with bounded 
principal curvatures. Then either it is flat, or its sectional curvature
is uniformly negative (bounded from above by a negative constant).
\end{prop}

Those results should be compared to the existence and uniqueness of a
maximal surface in a maximal globally hyperbolic AdS $3$-dimensional
manifold, see \cite{BBZ}. Theorem \ref{tm:existence} applies to this
case, with $S_0$ the lift of a closed surface in the globally hyperbolic
manifold $M$.  In this case the boundary at infinity of $S$ is the
limit set of $M$, which is the graph of a quasi-symmetric
homeomorphism (see \cite{mess,mess-notes}).  Theorem \ref{tm:qsym}
then shows that $w(\dr_\infty S)<\pi/2$, so that Theorem
\ref{tm:uniqueness} also applies.

\subsection{A characterization of quasi-symmetric homeomorphisms of the circle}

Consider a homeomorphism $u:S^1\rightarrow S^1$, let $\Gamma_u\subset S^1\times 
S^1\simeq \dr_\infty AdS_3$ be its graph.

\begin{theorem} \label{tm:qsym}
$w(\Gamma_u)$ is at most $\pi/2$. It is strictly less than $\pi/2$ if
  and only if $u$ is quasi-symmetric.
\end{theorem}

The first part here is just Lemma \ref{lm:width}, already mentioned above. 
The second part is proved in Section \ref{ssc:qsym}.

This statement can be considered in a purely projective way, because the
fact that a point of $\dr_-C(\Gamma_u)$ is at distance strictly less than
$\pi/2$ from $\dr_+C(\Gamma_u)$ corresponds to a purely projective 
property, stated in terms of the duality between points and space-like
planes in $AdS_3$, see Section \ref{ssc:causal}. 
This duality is itself a projective notion, see Section \ref{proj:sec}.

The proof uses the considerations explained above on the properties of
maximal surfaces bounded by $\Gamma_u$, it can be found in Section 
\ref{ssc:qsym}. It is based on Theorem \ref{tm:existence2} and to a partial
converse, in dimension 3 only: if an
acausal graph in $\dr_\infty AdS_3^*$
is the boundary of a maximal surface with bounded second fundamental form
which is not a ``horosphere'' (as described in Section \ref{ssc:52}),
then $\Gamma$ is the graph of a quasi-symmetric homeomorphism from 
$S^1$ to $S^1$.

\subsection{What follows}

Section \ref{sc:ads} contains a number of basic notions on the anti-de Sitter
(AdS) space and some of it basic properties. It is included here for completeness,
in the hope of making the paper reasonably self-contained for reader not
yet familiar with AdS geometry. Section \ref{sc:spacelike} similarly contains
some basic facts (presumably less well-known) on space-like hypersurfaces
in the AdS space. 

Section \ref{sc:maximal} is perhaps the heart of the paper. After some
preliminary statements on maximal space-like hypersurfaces in AdS, it contains
both an existence theorem for maximal hypersurfaces with given boundary
data at infinity, and a statement on the regularity of those hypersurfaces
under a geometric condition on the boundary at infinity. 
This condition
is later translated (for surfaces in the 3-dimensional AdS space) in terms
of quasi-symmetric regularity of the data at infinity.

In Section \ref{sc:uniqueness} we further consider this regularity issue,
with emphasis on surfaces in $AdS_3^*$, and we prove a uniqueness 
result for maximal surfaces with regular enough data at infinity. Finally
we prove Theorem \ref{tm:ml}.

Appendix A contains an alternative proof of the existence of a maximal
hypersurface with given data at infinity, based on the mean curvature 
flow. This approach also yields some regularity results.

\section{The Anti de Sitter space}
\label{sc:ads}

This section contains a number of basic statements on AdS geometry, which are
necessary in the proof of the main results. Readers who are already familiar
with AdS geometry will find little interest in it, we have however decided
to include it to make the paper self-contained, hoping that it is useful for 
readers interested in Teichm\"uller theory but not yet in AdS geometry.

\subsection{Definitions} \label{ssc:ads}

We consider the hyperbolid model of the hyperbolic space: the
hyperbolic space $\HH^n$ is identified with the set of future-pointing
unit timelike vectors in $(n+1)$-dimensional Minkowski space
$\R^{n,1}$.  In this work, if it is not specified differently, we
always use this identification. In particular points of $\HH^n$ are
identified with elements $(x_1,\ldots,x_{n+1})\in\R^{n+1}$ such that
$\sum_1^nx_i^2-x_{n+1}^2=-1$.  We also fix the point
$x^0=(0,\ldots,0,1)\in\HH^n$.

Let $\R^{n,2}$ be $\R^{n+2}$ equipped with the symmetric $2$-form
\[
  \langle x,y\rangle=x_1y_1+\ldots+x_ny_n-x_{n+1}y_{n+1}-x_{n+2}y_{n+2}
\]
The $(n+1)$-dimensional anti de Sitter space is the set
\[
  AdS_{n+1}^*=\{x\in\R^{n,2}|\langle x,x\rangle=-1\}~.
\]

The tangent space at a point $x\in AdS_{n+1}^*$ is the linear
hyperplane orthogonal to $x$ with respect to
$\langle\cdot,\cdot\rangle$. The restriction of
$\langle\cdot,\cdot\rangle$ to $T_xAdS_{n+1}^*$ is a Lorentzian scalar product.

\begin{remark} \label{rk:basics}
With this definition of $AdS^*_{n+1}$, its isometry group is immediately seen
to be $O(n,2)$. In particular, this isometry group acts transitively on the
points of $AdS^*_{n+1}$. More precisely, it acts simply transitively on the
set of couples $(x,e)$ where $x\in AdS^*_{n+1}$ and $e$ is an orthonormal 
basis of $T_xAdS^*_{n+1}$. It is also clear (using the action of $O(n,2)$ by
isometries) that the geodesics in $AdS^*_{n+1}$ are precisely the intersections
of $AdS^*_{n+1}$ with the linear planes in $\R^{n,2}$ containing $0$.   
\end{remark}

There is a map
\[
 \Phi: \HH^n\times\R\rightarrow AdS_{n+1}^*
\]
defined by 
\begin{equation}\label{covering:eq}
\Phi((x_1,\ldots,x_{n+1}),t)=(x_1,x_2,\ldots, x_n,x_{n+1}\cos t,x_{n+1}\sin t)\,.
\end{equation}
$\Phi$      is     a      covering      map,     so      topologically
$AdS_{n+1}^*\cong\HH^n\times  S^1$.  It will often be convenient to
consider the universal cover $AdS_{n+1}$  of $AdS_{n+1}^*$,  that is
$\HH^n\times\R$,  equipped  with  the   pull-back  of  the  metric  on
$AdS_{n+1}^*$.

It is easy to see that this metric at a point
$((x_1,\ldots,x_{n+1}),t)$ takes the form
\begin{equation}\label{adsmetric1:eq}
    g_{\HH}-x_{n+1}^2 dt^2~.
\end{equation}

If we consider  the Poincar\'e model of $\HH^n$, the metric
can be written as
\begin{equation}\label{adsmetric3:eq}
  \frac{4}{(1-r^2)^2}(dy_1^2+\ldots+
  dy_n^2)-\left(\frac{1+r^2}{1-r^2}\right)^2 dt^2~,
\end{equation}
where $r=\sqrt{y_1^2+\ldots+y_n^2}$ and $y_1,\ldots, y_n$ are the
Cartesian coordinates on the ball $\{y\in\R^n| r(y)<1\}$.

By (\ref{adsmetric1:eq}) we see that the time translations
\[
    (x,t)\rightarrow(x,t+a)
\]
are isometries of $AdS_{n+1}$. The coordinate field
$\frac{\partial\,}{\partial t}$ is a Killing vector field and the
slices $\mathbb H^n\times\{t\}$ are totally gedesic.

We denote by $\bar\nabla$ the Levi-Civita connections of both
$AdS_{n+1}$ and $\mathbb H^n$.  Since $\mathbb H^n\times\{t\}$ is
totally geodesic, the restriction of $\bar\nabla$ on this slice
coincides with its Levi-Civita connection.

We say that a vector $v\in T_{x,t}AdS_{n+1}$ is horizontal if it is tangent to
the slice $\mathbb H^n\times\{t\}$.  Analogously it is vertical if it
is tangent to the line $\{x\}\times\mathbb R$.

The lapse function $\phi$ is defined by
\[
  \phi^2=-\left\langle \frac{\partial\,}{\partial t},
  \frac{\partial\,}{\partial t}\right\rangle
\]
The gradient of $t$ is a vertical vector at each point and it equal to
\[
  \bar\nabla t=-\frac{1}{\phi^2}\frac{\partial\,}{\partial t}~,
\]
so its squared norm is equal to $-\frac{1}{\phi^2}$.



\subsection{The asymptotic boundary and the causal structure}

We  denote by
$\overline{AdS}_{n+1}$ the manifold with boundary 
$\overline{\mathbb H^n}\times\mathbb R$, where $\overline{\mathbb H}^n$ is
the usual compactification of $\mathbb H^n$ (obtained for instance in 
the projective model of $\mathbb H^n$). Another way to consider 
$\overline{AdS}_{n+1}$ is as the universal cover of the compactification
of $AdS^*_{n+1}$ defined by adding the projectivization of the cone of
vectors $x\in \R^{n,2}$ such that $\langle x,x\rangle=0$.

Clearly $AdS_{n+1}$ is the interior part of $\overline{AdS}_{n+1}$, whereas
its boundary, $\partial\mathbb H^n\times\mathbb R$
 is called the \emph{asymptotic boundary} of $AdS_{n+1}$ and is
 denoted by $\partial_\infty AdS_{n+1}$. The following statement
is clear when considering the definition of $AdS^*_{n+1}$ as a quadric.

\begin{lemma} \label{isoextension:lem}
Every isometry $f$ of $AdS_{n+1}$ extends to a homeomorphism
of $\overline{AdS}_{n+1}$.
\end{lemma}

The asymptotic boundary of  a set $K\subset AdS_{n+1}$
--- denoted by $\partial_\infty K$ ---
is the set of the accumulation points of $K$ in $\partial_\infty AdS_{n+1}$.

By (\ref{adsmetric3:eq}) it is clear that the conformal structure on $AdS_{n+1}$
 extends to the boundary.
This means that in the conformal class of the metric $g$
 there is a metric $g^*$ that extends to the 
boundary. We can for instance put $g^*=\frac{1}{\phi^2}g$.

A vector $v$ tangent at some point in $\partial_\infty AdS_{n+1}$ is
 timelike (ligthlike, spacelike) if $g^*(v,v)<0$ ($=0$,
$>0$). Notice that the definition makes sense since the sign of
$g^*(v,v)$ depends only on the conformal class of $g^*$.

\begin{lemma}
Let $c:(-1,1)\rightarrow AdS_{n+1}$ be an inextensible timelike
path. If the function $t$ is bounded from above on $c$, there exists
the limit $p_1=\lim_{s\rightarrow 1}c(s)\in \partial_\infty
AdS_{n+1}$.
\end{lemma}
\begin{proof}
 The vertical component of $\dot c$ is
\[
   \dot c_V=\langle \dot c,\bar\nabla
   t\rangle\frac{\partial\,}{\partial t}=\dot
   t\frac{\partial\,}{\partial t}~.
\]
Since the norm of $\frac{\partial\,}{\partial t}$ for $g^*$is $1$, we
have $|\dot c_V|_{g^*}=\dot t$. On the other hand, the fact that $c$
is timelike implies
\[
  |\dot c_H|_{g^*}\leq |\dot c_V|_{g^*}=\dot t\,.
\]
Since the function $t$ is increasing along $c$, the bound on $t$ along
$c$ implies that $\dot c$ is bounded in a neighbourhood of $1$.  It
follows that the path $c_H$ obtained by projecting $c$ to $\mathbb
H^n$, has finite length with respect to the metric
$\frac{1}{\phi^2}g_{\mathbb H}$.  This implies that there exists the
limit $x_1=\lim_{s\rightarrow 1}c_H(s)$.  On the other hand, since $t$
is increasing along $ c$ there exists the limit
$t_1=\lim_{s\rightarrow 1}t(c(s))$.  The point $p_1=(x_1,t_1)$ is the
limit point of $c$.  Since we assume that $c$ is inextensible in $AdS_{n+1}$,
$p_1\in\partial_\infty AdS_{n+1}$.
\end{proof}

The point $p_1$ is an asymptotic end-point of $c$.

An inextensible path is  without end-points if and only if the function
$t$ takes all the real values along $c$, or equivalently, if $c$ does not admit
any asymptotic end-point.
Vertical lines are instances of inextensible paths without end-points.

\subsection{Geodesics and geodesic hyperplanes in $AdS_{n+1}$}
\label{geodesics:sec}

The next statement, which is classical, describes the geodesics in
$AdS_{n+1}^*$, considered as a quadric in $\R^{n,2}$.

\begin{lemma}[see \cite{benedetti-bonsante}]\label{geoimmersed:lem}
Geodesics in $AdS_{n+1}^*$ are the intersection $AdS_{n+1}^*$ with
linear $2$-planes in $\R^{n,2}$ containing $0$.

In particular given a tangent vector $v$ at some point $p\in
AdS_{n+1}^*$ we have
\begin{equation}\label{geoformula:eq}
   \exp_p(sv)=\left\{\begin{array}{ll} \cos (s) p+\sin (s) v & \textrm{if
   } \langle v,v\rangle=-1\,;\\ p+sv &\textrm{if }\langle v,v\rangle
   =0\,;\\ \ch (s) p+\sh (s) v &\textrm{if } \langle v,v\rangle=1\,.
   \end{array}\right.
\end{equation}
\end{lemma}

\begin{remark}
Totally geodesics $k$-planes in $AdS^*_{n+1}$
are the intersection of $AdS^*_{n+1}$ with
$(k+1)$-linear planes of $\mathbb R^{n,2}$ containing $0$.
\end{remark}




Spacelike and lightlike geodesics are open simple curves. Homotopically,
timelike geodesics are simple closed non-trivial curve.  Moreover
every complete timelike geodesic starting at $p$ passes through $-p$
at time $(2k+1)\pi$ and at $p$ at time $2k\pi$ for $k\in\mathbb Z$.

Passing to the universal cover, we get the following statement.

\begin{lemma}\label{geodesics:lem}
Given a point $p=(x,t)\in AdS_{n+1}$ there is a discrete set
$\{p_n|n\in\Z\}$ such that every timelike geodesics $\gamma$ starting
through $p$ passes through $p_k$ at time $t=n\pi$.

Moreover, $p_{2k}=(x,t+2k\pi)$ and $p_{2k+1}=(y,t+(2k+1)\pi)$ where
$y$ is some point in $\mathbb H^n$ independent of $k$. 
\end{lemma}

In what follows we will often use the points $p_1$ and $p_{-1}$.
To simplify the notation we will denote these points 
by $p_+$ and $p_-$.

Timelike geodesics are timelike paths without end-points.  On the
other hand since spacelike geodesics are conjugated to horizontal ones
by some isometry, they have $2$ asymptotic end-points.  
Using the projection $\Phi$ one can check that the path $c(s)=
\left(x(s),\mathrm{arccos}\left(\frac{1}{\sqrt{1+s^2}}\right)\right)$
where $x(s)=\left(s,0,\ldots,0,\sqrt{1+s^2}\right)$ is a lightlike geodesic.
Since $c$ has two asymptotic end-points,
the same property holds for every lightlike geodesic.

\begin{remark}\label{asym:rk}
Points in $\partial_\infty
AdS_{n+1}$ related by a timelike arc in $\partial_\infty AdS_{n+1}$
are not joined by a geodesic arc in $AdS_{n+1}$.
Indeed by the above description if a geodesic connects two points in
the asymptotic boundary of $AdS_{n+1}$ then it is either space-like or
light-like (and in this case it is contained in the boundary).
\end{remark}

Totally geodesic $n$-planes in $AdS_{n+1}$ are
distinguished by the restriction of the
ambient metric on them. They can be timelike, spacelike or lightlike
according as whether this restriction has Lorentzian, Euclidean or degenerate
signature.

Spacelike hyperplanes are conjugated by some isometry
to horizontal planes. Timelike hyperplanes
are conjugated by some isometry to the hyperplane
$P_0\times\mathbb R$, where $P_0$ is a totally geodesic
hyperplane in $\HH^n$.

For lightlike hyperplanes we will need a more precise description.

\begin{lemma}\label{llkplanes:lem}
Let $P$ be a lightlike hyperplane.  There are two points $\zeta_-$ and $\zeta_+$ in
$\partial_\infty AdS_{n+1}$ such that $P$ is foliated by lightlike
geodesics with asymptotic end-points $\zeta_-$ and $\zeta_+$.

The foliation of $P$ by lightlike geodesics extends to a a foliation
of $\overline P\setminus\{\zeta_-,\zeta_+\}$ by lightlike geodesics, where
$\overline P$ denotes the closure of $P$ in $\overline{AdS}_{n+1}$.
\end{lemma}

\begin{proof}
It is sufficient to prove the statement for a specific lightlike
plane.  Consider the hypesurface $P_0=\left\{(x,t)\in AdS_{n+1}|
t=\mathrm{arcsin}\left(\frac{x_1}{x_{n+1}}\right)\right\}$.
Using the projection $\Phi$ one see that $P_0$ is a totally geodesic plane,
indeed $\Phi(P_0)$ is a connected component of the intersection of $AdS_{n+1}^*$
with the linear plane defined by the equation $y_1-y_{n+2}=0$.

We consider the natural parameterization 
$\sigma:\mathbb H^n\rightarrow P_0$
defined by $\sigma(x)=(x, \mathrm{arcsin}(\frac{x_1}{x_{n+1}}))$.  Since the
function $\frac{x_1}{x_{n+1}}$ extends to the boundary of $\mathbb H^n$, 
the map $\sigma$ extends to $\overline{\HH}^n$ and
gives a parameterization of the closure $\overline P_0$ of $P_0$ in
$\overline{AdS}_{n+1}$.

The level surfaces $H_a=\{\frac{x_1}{x_{n+1}}=a\}$ are totally
geodesic hyperplanes orthogonal to the geodesic
$c=\{x_2=\ldots=x_n=0\}$. Let $N$ be the unit future-oriented vector
field on $\mathbb H^n$ orthogonal to $H_a$ for all $a$.  A simple
computation shows that
\begin{itemize}
\item for all $a$, $\sigma|_{H_a}$ is an isometric embedding;
\item $\hat N=\sigma_*(N)$ is a lightlike field;
\item $\hat N$ is orthogonal to $\sigma(H_a)$.
\end{itemize}
It follows that $P_0$ is a lightlike plane.
The integral lines of $\hat N$ produce a foliation of $P_0$ by
lightlike geodesics.  Notice that integral lines of $\hat N$ are the
images of integral lines of the field $N$.  By standard hyperbolic
geometry, all these lines join the endpoints, say $x_-$, $x_+$, of the
geodesic $c$. We conclude that lightlike geodesics of $P_0$ join
$\sigma(x_-)$ to $\sigma(x_+)$.
Since the foliation of $\mathbb H^n$ by integral lines of $N$ extends
to a foliation of $\overline\HH^n\setminus\{x_-,x_+\}$, the foliation
given by $\hat N$ extends to a foliation of
$\overline{P_0}\setminus\{\zeta_-,\zeta_+\}$.  By continuity we conclude that
the leaves of this foliation are lightlike.
\end{proof}

For a lightlike plane $P$ the points $\zeta_-$ and $\zeta_+$ are called respectively
the past and the future end-points of the plane.


Spacelike and lightlike hyperplanes
disconnect $AdS_{n+1}$ in two connected components,
that coincide with the past and the future of them.
Their asymptotic boundary is a no-where timelike closed curve.
On the other hand the asymptotic boundary of
a timelike plane is the union of two inextensible timelike curves.

\subsection{The causal structure of $AdS_{n+1}$}
\label{ssc:causal}

If $c:[0,1]\rightarrow AdS_{n+1}$ is a timelike path, its length is
defined in this way:
\[
\ell(c)=\int_0^1\left(-\langle\dot c(s),\dot c(s)\rangle\right)^{1/2} ds\,.
\] 

Given $p\in AdS_{n+1}$ we consider the set $P_-(p)$ (resp. $P_+(p)$)
defined respectively as the set of points that can be joined to $p$
through a past-directed (resp. future-directed) timelike geodesic of
length $\pi/2$.

\begin{figure}
\begin{center}
\input causalityads.pstex_t
\end{center}
\end{figure}

\begin{remark}\label{planes:rem}
For a point $x\in AdS^*_{n+1}$ we can identify the set of unit
timelike vectors at $x$ with the geodesic plane $P^*_x=x^\perp\cap
AdS^*_{n+1}$ (where $x^\perp$ is the linear plane orthogonal to $x$).
$P^*_x$ has two connected components. Equation (\ref{geoformula:eq}) shows
that these components are the images of $P_+(p)$ and $P_-(p)$,
where $p$ is any preimage of $x$ in $AdS_{n+1}$.
\end{remark}

The following properties of $P_-(p)$ and $P_+(p)$ are a direct consequence
of Remark \ref{planes:rem}

\begin{lemma}
The sets $P_-(p)$ and $P_+(p)$ are complete, space-like totally
geodesic planes.  Every timelike geodesic starting at $p$ meets
$P_-(p)$ and $P_+(p)$ orthogonally.
\end{lemma}

\begin{remark}
For the point $p_0=(x^0,0)$, a direct computation
(still using the projection $\Phi$) shows that
$P_-(p_0)$ and $P_+(p_0)$ are level curves
of the time function $t$ corresponding to values $-\pi/2$
and $\pi/2$ respectively. 
\end{remark}

The planes $P_-(p)$ and $P_+(p)$ are disjoint and bound an open precompact
domain $U_p$ in $\overline{AdS}_{n+1}$.  For instance, for $p=(x^0,0)$
we have $U_p=\{(x,t)\in\overline{AdS}_{n+1}| -\pi/2<t<\pi/2\}$.

By definition the interior of $U_p$ (denoted by $\mathrm{int}(U_p)$) is the
intersection of $U_p$ with $AdS_{n+1}$. Notice that
\[
\mathrm{int}(U_p)= I^+(P_-(p))\cap I^-(P_+(p))\,.
\]

Notice that $P_+(p_k)=P_-(p_{k+1})$ for every $k\in\Z$.
In particular $U_{p_i}\cap U_{p_j}=\emptyset$ if $|i-j|>1$ 
and $\overline{U_{p_i}}\cap \overline{U_{p_{i+1}}}=P_+(p_i)$. 

Given $p\in AdS_{n+1}$ we denote by $C_p$ the set of points joined to
$p$ through a timelike geodesic of length less than $\pi/2$.

\begin{prop}\label{caus:prop}
\begin{itemize}
\item $C_p\subset U_p$.
\item spacelike and lightlike geodesics join $p$ to points in $U_p\setminus C_p$, whereas
timelike geodesics are contained in
$\bigcup_{n\in\Z} C_{p_n}$.
\item $I^+(p)\subset C_p\cup I^+(P_+(p))=C_p\cup \bigcup_{k>0} U_{p_k}$.
\item $\partial C_p\cap U_p$ is the lightlike cone through $p$, whereas $\partial_\infty C_p$
is the union of the asymptotic boundary of $P_+(p)$ and the asymptotic boundary
of $P_-(p)$.
\end{itemize}
\end{prop}

This proposition can be easily proved using the projection $\Phi$ and the explicit
formula (\ref{geoformula:eq}).

It is worth noticing that $AdS_{n+1}$ is not geodesically convex. 
Indeed the set of points in $AdS_{n+1}$ that can be joined to $p$ by
a geodesic is $\mathrm{int}(U_p)\cup\bigcup C_{p_k}$.

\begin{cor}
The set  $I^-(p_+)\cap I^+(p_-)$ is the maximal star neighbourhood of $p$.
\end{cor}

Given $p\in AdS_{n+1}$ and $q\in I^+(p)$, the distance between them is
defined as
\[
  \delta(p,q)=\sup\{\ell(c)|c\textrm{ timelike path joining }p\textrm{
    to }q\}\,.
\]
The next statement is true in a rather general context and 
can be proved by classical arguments.

\begin{lemma}
If $U$ is a star neighbourhood of $p$, then the distance from $p$
\[
  \delta_p:U\cap I^+(p)\ni q\mapsto\delta(p,q)\in\mathbb R
\]
is smooth.
For $q\in U\cap I^+(p)$ the distance $\delta(p,q)$ is realized by the unique
geodesic joining $p$ to $q$ contained in $U$.
\end{lemma}

\begin{remark}
The definition of the distance shows that 
for $q\in I^+(p)\cap U$ and $r\in I^+(q)$, the reverse of the triangle
inequality holds
\begin{equation}\label{reverse:eq}
  \delta(p,r)\geq\delta(p,q)+\delta(q,r)\,.
\end{equation}
\end{remark}

\subsection{The projective model}
\label{proj:sec}

As noted in the proof of Lemma \ref{geodesics:lem} the geodesics
in $AdS_{n+1}^*$ are obtained as the intersection
of $AdS_{n+1}^*$ with the linear planes of $\mathbb R^{n+2}$ containing $0$.

For this reason the projection map
\[
   \pi:AdS_{n+1}\rightarrow \mathbb R P^{n+1}
\]
is projective: it sends geodesics of $AdS_{n+1}$ 
to projective segments. The image of this projective map is
the interior of a quadric $Q\subset \R P^{n+1}$ of signature 
$(n-1,1)$. 

Notice for $p\in AdS_{n+1}$ the domain $\Phi(\mathrm{int}(U_p))$ 
is a connected component
of $AdS^*_{n+1}\setminus P^*_{\Phi(p)}$.
Thus the domain $\pi(\Phi(U_p))$ is contained in some affine chart 
of $\mathbb R P^{n+1}$.

In this way we construct a projective embedding
\[
\pi^*: \mathrm{int}(U_p)\rightarrow\mathbb R^{n+1}~.
\]
The map $\pi^*$ can be easily computed assuming $p=(x^0,0)$.  In this
case $U_p=\{(x,t)| t\in (-\pi/2,\pi/2)\}$ so
$\Phi(\mathrm{int}(U_p))=\{(y_1,\ldots, y_n, ,y_{n+1},y_{n+2})\in
AdS^*_{n+1}|y_{n+1}>0\}$ and
\begin{equation}\label{projective:eq}
  \pi^*(x_1,\ldots,x_{n+1}, t)= \left(\frac{x_1}{x_{n+1}\cos
    t},\frac{x_2}{x_{n+1}\cos t},\ldots,\frac{x_n}{x_{n+1}\cos t},\tan
  t\right)
\end{equation}
for every $(x_1,\ldots,x_n)\in\mathbb H^n$ and $t\in (-\pi/2,\pi/2)$.

Notice that the map extends continuously on $U_p$ to a map, still
denoted by $\pi^*$.  From (\ref{projective:eq}), the image
$\pi^*(U_p)$ is the set
\begin{equation}\label{image:eq}
\{(z_1,\ldots,z_{n+1})|\sum_{i=1}^n z_i^2\leq z_{n+1}^2+1\}\,.
\end{equation}

In particular we deduce that every point $q\in U_p$ (even on the
boundary) can be joined to $p$ by a unique geodesic and that this
geodesic continuously depend on $q$.

We have seen above how to associate to a point $p\in AdS_{n+1}$ two
totally geodesic space-like hyperplanes $P_-(p)$ and $P_+(p)$. Both
planes are sent by $\pi$ to the intersection with $\pi(AdS_{n+1}^*)$
of the same projective plane $P$, and $P$ has a purely projective
definition. Indeed the light-cone of $p$ is tangent to $Q$ along a
circle $C$, and the image by $\pi$ of the boundary at infinity of
$P_-(p)$ is precisely $C$. One way to see this is by using the fact
that in the projective model of $AdS_{n+1}$ (as for the hyperbolic
space) the distance between two points can be defined in terms of the
Hilbert distance of the quadric $Q$, see e.g. \cite{shu}.

This duality extends to a duality between totally geodesic
(space-like) $k$-planes in $\pi(AdS_{n+1})$, with the dual of a
$k$-plane $P$ being a $(n-k)$-plane $P^*$. Then $P^*$ can be defined as
the intersection between the hyperplanes dual to the points of $P$,
and conversely. Then $P^*$ can be characterized as the set of points
at distance $\pi/2$ from $P$ along a time-like segment, and
conversely.

\subsection{The 3-dimensional AdS space} \label{ssc:3d}

The general description of the $n$-dimensional anti-de Sitter space
$AdS_{n+1}^*$ above can be refined when $n=2$, and $AdS_3^*$ has some
quite specific properties.

One such specificity is that $AdS_3^*$ is none other than the Lie group 
$SL(2,\R)$, with its Killing metric. This point of view, which is important
in itself (see \cite{mess,mess-notes}), will not be used explicitly here.

Another feature which is specific of $AdS_3$ is the fact that the boundary
of $\pi(AdS_3)$ in $\R P^3$ is a quadric of signature $(1,1)$ which, as is
well known, is foliated by two families of projective lines, which we
will call $\cL_l$ and $\cL_r$ ($l$ and $r$ stand for ``left'' and ``right''
here). Those 
projective lines correspond precisely to the isotropic curves in the
Lorentz-conformal structure on $\dr_\infty AdS_3$. Each line of one family
intersects each line of the other family at exactly one point, this 
provides an identification of $\dr \pi(AdS_3^*)$ with $S^1\times S^1$,
with each copy of $S^1$ identified with one of the two families of lines
foliating $\dr\pi(AdS_3^*)$.

This has interesting consequences, in particular those explained in 
Section \ref{ssc:diffeos}. Another consequence is that the isometry group of
$AdS_3$ can be naturally identified (up to finite index) with the product of
two copies of $PSL(2,\R)$. Indeed any isometry of $AdS_3$ in the connected
component of the identity acts on the two families of lines foliating
$\dr_\infty AdS_3$ by permuting those lines, and this action is projective
on each family of lines. Conversely, any couple of elements of $PSL(2,\R)$
can be obtained in this manner.

%% file: causalityads.pstex_t
\begin{picture}(0,0)%
\includegraphics{causalityads.pstex}%
\end{picture}%
\setlength{\unitlength}{1579sp}%
\begingroup\makeatletter\ifx\SetFigFont\undefined%
\gdef\SetFigFont#1#2#3#4#5{%
  \reset@font\fontsize{#1}{#2pt}%
  \fontfamily{#3}\fontseries{#4}\fontshape{#5}%
  \selectfont}%
\fi\endgroup%
\begin{picture}(14644,9076)(2675,-9429)
\put(5604,-5063){\makebox(0,0)[lb]{\smash{{\SetFigFont{5}{6.0}{\rmdefault}{\mddefault}{\updefault}{\color[rgb]{0,0,0}$p$}%
}}}}
\put(3825,-3606){\makebox(0,0)[lb]{\smash{{\SetFigFont{5}{6.0}{\rmdefault}{\mddefault}{\updefault}{\color[rgb]{0,0,0}$P_+(p)$}%
}}}}
\put(3782,-6670){\makebox(0,0)[lb]{\smash{{\SetFigFont{5}{6.0}{\rmdefault}{\mddefault}{\updefault}{\color[rgb]{0,0,0}$P_-(p)$}%
}}}}
\put(9739,-5020){\makebox(0,0)[lb]{\smash{{\SetFigFont{5}{6.0}{\rmdefault}{\mddefault}{\updefault}{\color[rgb]{0,0,0}$p$}%
}}}}
\put(9847,-2684){\makebox(0,0)[lb]{\smash{{\SetFigFont{5}{6.0}{\rmdefault}{\mddefault}{\updefault}{\color[rgb]{0,0,0}$p_+$}%
}}}}
\put(9675,-7292){\makebox(0,0)[lb]{\smash{{\SetFigFont{5}{6.0}{\rmdefault}{\mddefault}{\updefault}{\color[rgb]{0,0,0}$p_-$}%
}}}}
\put(17304,-7313){\makebox(0,0)[lb]{\smash{{\SetFigFont{5}{6.0}{\rmdefault}{\mddefault}{\updefault}{\color[rgb]{0,0,0}$I^-(p)$}%
}}}}
\put(2690,-5011){\makebox(0,0)[lb]{\smash{{\SetFigFont{5}{6.0}{\rmdefault}{\mddefault}{\updefault}{\color[rgb]{0,0,0}$U_p$}%
}}}}
\put(5280,-5531){\makebox(0,0)[lb]{\smash{{\SetFigFont{5}{6.0}{\rmdefault}{\mddefault}{\updefault}{\color[rgb]{0,0,0}$C_p$}%
}}}}
\end{picture}%

%% file: fb1.tex
\section{Spacelike graphs in $AdS_{n+1}$} \label{sc:spacelike}

This section continues the description of the geometry of the AdS
space, with emphasis on space-like surfaces. Readers already familiar
with AdS geometry might not be very surprised by most of the results,
but several notations and lemmas will be used in the next section.

\subsection{Definitions}

A smooth embedded hypersurface $M$ in $AdS_{n+1}$ is
\emph{spacelike} if for every $x\in M$ the restriction of
$\langle\cdot,\cdot\rangle$ on $T_xM$ is positive definite.  It turns
out that a Riemannian structure is induced on every spacelike
hypersurface by the ambient metric.

We say that a spacelike surface $M$ in $AdS_{n+1}$ is a \emph{graph}
if there is a function
\[
    u:\HH^n\rightarrow \R
\]
such that $M$ coincides with the graph of $u$.

First let us check which functions correspond to spacelike graphs.

The function $u$ induces a function on $\mathbb H^n\times\mathbb R$
\[
   \hat u(x,t)=u(x)\,.
\]
The gradient of $\hat u$ at a point $(x,t)$ is the horizontal vector
that projects to the gradient of $u$ at $x$.

The graph of $u$, say $M=M_u$, is defined by the equation $\hat
u-t=0$.  Thus the tangent space $T_{(x,u(x))}M=\ker (dt-d\hat u)_{(x,
  u(x))}$.  In particular the normal direction of $M$ at $(x,u(x))$ is
generated by the vector
\begin{equation}\label{normal_dir:eq}
    \bar\nu=\bar\nabla t-\bar \nabla \hat u
\end{equation}
whose norm is
\[
   |\bar\nabla \hat u|^2-\frac{1}{\phi^2}~.
\]
Since $|\bar\nabla \hat u|=|\bar\nabla u|$ we deduce that $M$ is
spacelike if and only if
\begin{equation}\label{splike:eq}
  1-\phi^2|\bar\nabla u|^2<0~,
\end{equation}
and the future-pointing normal vector is
\begin{equation}\label{normal:eq}
\nu=\frac{\phi}{\sqrt{1-\phi^2|\bar\nabla u|^2}}(\bar\nabla \hat u-\bar\nabla t)\,.
\end{equation}

It is interesting to express (\ref{splike:eq}) using the Poincar\'e
model of hyperbolic space.  In that case we have
\[
     \bar\nabla u=\frac{(1-r^2)^2}{4}\left(\frac{\partial u}{\partial
       y_1},\ldots,\frac{\partial u}{\partial y_n}\right)
\]
so
\[
  |\bar\nabla u|^2=\frac{(1-r^2)}{4}\sum\left(\frac{\partial u}{\partial y_j}\right)^2
\]     
and condition (\ref{splike:eq}) becomes
\begin{equation}\label{splike2:eq}
   \sum_j\left(\frac{\partial u}{\partial y_j}\right)^2< \frac{4}{(1+r^2)^2}\,.
\end{equation}
In particular the function $u$ is $2$-Lipschitz with respect to the
Euclidean distance of the ball.  

\begin{lemma} \label{lm:boundary}
Let $M=M_u$ be a smooth spacelike graph in $AdS_{n+1}$.
Then the function $u$ extends to a continuous function
\[
  \bar u:\bar\HH^n\rightarrow\R\,.
\]
In particular the closure of $M$ in $\overline{AdS}_{n+1}$ is still a graph.
\end{lemma}

\subsection{Acausal surfaces}
\label{ssc:acausal}

A $\mathrm{C}^{0,1}$ hypersurface $M$ in $AdS_{n+1}$ is said to be weakly spacelike
if for every $p\in M$ there is a neighbourhood $U$ of $p$ in $AdS_{n+1}$ such that
$U\setminus M$ is the disjoint union $I^+_U(M)\cup I^-_U(M)$.

A neighbourhood satisfying the above property will be called a good
neighbourhood of $p$.

It is not hard to see that a spacelike surface is weakly spacelike. On
the other hand a $\mathrm C^1$ weakly spacelike surface is
characterized by the property that no tangent plane is timelike.

A weakly spacelike graph is  a  weakly spacelike surface that is the graph
of some function $u$.
Weakly spacelike graphs correspond to Lipschitz functions  $u$ 
such that the inequality
\[
  1-\phi^2|\bar\nabla u|^2\leq 0
\]
holds almost everywhere.

As for spacelike graphs it is still true that the closure of acausal
graphs in $\overline{AdS}_{n+1}$ is a graph.

First we provide an intrinsic characterization of weakly spacelike graphs.

\begin{prop}\label{spacelikegraphs:prop}
Let $M$ be a connected weakly spacelike   hypersurface. 
The following statements are equivalent:
\begin{enumerate}
\item $M$ is a weakly spacelike graph;
\item $AdS_{n+1}\setminus M$ is the union of $2$ connected components;
\item every inextensible timelike curve without end-points
meets $M$ exactly in one point.
\end{enumerate}
\end{prop}

\begin{proof}
The implication $(1)\Rightarrow (2)$ is clear.

Assume $(3)$ holds. Then every vertical line meets $M$ exactly in one point.
This shows that the projection $\pi:M\rightarrow\mathbb H^n$ is one-to-one.
Since $M$ is a topological manifold, the Invariance of Domain Theorem 
implies that $\pi$ is a homeomorphism. Thus $M$ is a graph.

Finally suppose that $(2)$ holds.  We consider the equivalence
relation on $M$ such that $p\sim q$ if there are good neighbourhoods
$U$ and $V$ of $p$ and $q$ respectively such that $I^+_U(p)$ and
$I^+_V(q)$ are contained in the same component of $AdS_{n+1}\setminus
M$.  Equivalence classes are open. Since $M$ is connected, all points
are equivalent.  We deduce that there is a component, say $\Omega_+$,
of $AdS_{n+1}\setminus M$ such that if $c=c(s)$ is a future-directed
timelike path hitting $M$ for $s=0$, then there is $\epsilon>0$ such
that $c(s)\in\Omega_+$ for $0<s<\epsilon$.  In the same way, there is
a component, say $\Omega_-$ such that $c(s)\in\Omega_-$ for
$-\epsilon<s<0$.

If $U$ is a good neighbourhood of some point $p\in M$, then 
$U\subset\Omega_+\cup M\cup\Omega_-$, so
 $\Omega_+\cup\Omega_-\cup M$ is an open  neighbourhood of $M$. 
Since the closure of every component of $AdS_{n+1}\setminus M$ contains
points in $M$, by the assumption $(2)$, 
$\Omega_+$ and $\Omega_-$  are 
different components of $AdS_{n+1}\setminus M$ and $AdS_{n+1}=\Omega_-\cup M\cup\Omega_+$.

It follows that no future-directed timelike curve starting at a point of $\Omega_+$ can end
at some point of $M$. Since any future-directed timelike curve that starts on $M$ intersects 
$\Omega_+$, points of $M$ are not related by timelike curves and
$I^+(M)\subset\Omega_+$ and $I^-(M)\subset\Omega_-$.

In particular, given a point $p\in M$, the surface $M$ is contained in $U_p$.
It follows that the restriction of the time-function $t$ on $M$
 is bounded in some interval $[a,b]$.
 Moreover $\Omega_+$ contains the region $\{(x,t)|t>b\}$, instead
 $\Omega_-$ contains the region $\{(x,t)|t<a\}$.
 
 Since the restriction of $t$ on any inextensible timelike curve without end-points 
 $c$ takes all the values
 of the interval $(-\infty,+\infty)$ we have that $c$ contains points of $\Omega_-$ and points
 of $\Omega_+$. Thus it must intersect $M$.
 Since points of $M$ are not related by timelike arcs, such intersection point is unique.
\end{proof}

\begin{remark}
Proposition \ref{spacelikegraphs:prop} implies that spacelike graphs
are intrinsically described in terms of the geometry of $AdS_{n+1}$.
In particular, if $M$ is a spacelike graph, and $\gamma$ is an
isometry of $AdS_{n+1}$, then $\gamma(M)$ is still a spacelike graph.
\end{remark}

\begin{remark}
Given a point $p\in AdS_{n+1}$ we have that $\partial I^+(p)$ is a
weakly spacelike graph.  Indeed we can assume $p=(x^0,0)$. In that
case it turns out that $\partial I^+(p)$ is the graph of the function
$\mathrm{arccos}\left(\frac{1}{x_{n+1}}\right)$.
\end{remark}

An important feature of weakly spacelike graphs is that they are
acausal as the following proposition states.

\begin{prop}\label{acausal:prop}
Let $M=M_u$ be a weakly spacelike graph in $AdS_{n+1}$, and let $\overline
M$ denote its closure in $\overline{AdS}_{n+1}$.
Given $p\in M$, then, for every $q\in\overline M$, $p$ and $q$
are connected by a geodesic $[p,q]$ that is not timelike.
Moreover, if this geodesic is lightlike, then it is contained in $M$.
\end{prop}

\begin{proof}
Proposition \ref{spacelikegraphs:prop} implies that $M\cap
I^+(p)=\emptyset$ and $M\cap I^-(p)=\emptyset$. In particular,
$M\subset U_p$ that is a star-neighbourhood of $p$.
It follows that any point $q$ of $\overline M$ is connected to $p$ by some geodesic
that continuously depends on $p$. Since points of $M$
cannot be connected to $p$ by a timelike geodesic, the same holds for
points in $\partial_\infty M$.

Finally, let us prove that if $[p,q]$ is lightlike, then it is contained in $M$.

Let $u_\pm:\mathbb H^n\rightarrow\mathbb R$ be such that $\partial
I^\pm(p)$ is the graph of $\Gamma_{u_\pm}$.  Let us set $p=(x_0,t_0)$
and $q=(x_1,t_1)$.  Consider the geodesic arc of $\mathbb H^n$, say
$x(s)$, starting from $x_0$ and ending at $x_1$ defined for $s\in
[0,T]$ ($T$ can be $+\infty$ if $x_1\in\partial\mathbb H^n$). Notice
  that the function of $s$ defined by $u_+(s)=u_+(x(s))$ satisfies
\begin{equation}\label{sys:eq}
\dot u_+=\frac{1}{\phi(x(s))}~,\qquad u_+(0)=t_0\,.
\end{equation}
On the other hand the function $u(s)=u(x(s))$ satisfies
\begin{equation}\label{sys2:eq}
\dot u=\langle\bar\nabla u,\frac{dx}{ds}\rangle\leq\frac{1}{\phi(x(s))}~,\qquad u(0)=t_0~.
\end{equation}
Comparing (\ref{sys:eq}) and (\ref{sys2:eq}) we deduce that
\[
   u(s)\leq u_+(s)~,
\]
and the equality holds at some $s_0$ if and only if $\dot
u(s)=\frac{1}{\phi(x(s))}$ on the interval $[0,s_0]$, that is
equivalent to say that the light-like segment joining $p=(x_0,t_0)$ to
$q=(x(s_0), u(x(s_0))$ is contained in $M$.

In an analogous way we show that $u_-(s)\leq u(s)$.
\end{proof}

\begin{remark}
The hypothesis that $M$ is a graph is essential in Proposition \ref{acausal:prop}.
It is not difficult to construct a spacelike surface $M$ containing points
$p$, $q$ that are related by a vertical segment.
\end{remark}

For a weakly spacelike surface $M$, a  point $p\in M$ is \emph{singular} 
if it is contained in the interior of some lightlike segment contained in $M$.
The singular set of $M$ is the set of singular points.

Analogously we define the singular set of the asymptotic boundary
$\Sigma$ of $M$.  Notice that the singular set of $\Sigma$ can be
non-empty even if $M$ does not contain singular points.





\subsection{The domain of dependence of a spacelike graph}

Let $M$ be a spacelike graph in $AdS_{n+1}$, and let
$\Sigma$ denote its asymptotic boundary.  We will suppose 
that $M$ does not contain any singular point.

The domain of dependence of
$M$ is the set $D$ of points $x\in AdS_{n+1}$ such that 
\emph{every inextensible causal path through $x$ intersects $M$.}

It can be easily shown that this property is equivalent to requiring
that $(I^+(x)\cup I^-(x))\cap M$ is precompact in $AdS_{n+1}$.

There is an easy characterization of $D$ in terms of $\Sigma$.

\begin{lemma}\label{dom:lem}
With the notations of Section \ref{geodesics:sec}, a point $p$ lies in $D$
if and only if $\Sigma$ is contained in  $U_p$.
\end{lemma}

\begin{proof}
Suppose that $p\in D$. Without loss of generality we can suppose that
$p\in I^-(M)$. By the hypothesis, $I^+(p)\cap M$ is precompact in
$AdS_{n+1}$ (whereas $I^-(p)\cap M=\emptyset$). Thus there is a compact
ball $B\subset\mathbb H^n$ such that $I^+(p)\cap M$ is contained in the cylinder
above $B$. In particular, $M\setminus (B\times\mathbb R)$ is
contained in $U_p$. It follows that $\Sigma\subset\overline{U_p}$.

If some point $x$ of $\Sigma$ were contained in $\partial_\infty P_+(p)$
then the geodesic joining $p$ to $x$ would be lightlike
and would intersects $M$ in some point $q$.
Then by Proposition \ref{acausal:prop}, the lightlike
geodesic segment joining $q$ to $x$ would be contained 
in $M$ and this would contradict the hypothesis that $M$ does not
contain any singular point.

Let us consider now  a point $p$ such that $\Sigma\subset U_p$.
Again we can suppose that $p\in I^-(M)$.
By the assumption the asymptotic boundary of $M$ and the asymptotic 
 boundary of $I^+(p)$ are disjoint.
It follows that $I^+(p)\cap M$ is pre-compact in $AdS_{n+1}$. 
\end{proof}

\begin{cor}
Two spacelike surfaces share the boundary at infinity
if and only if their domains of dependence coincide.
\end{cor}

\begin{prop} \label{boundarydep:prop} 
The domain $D$ is geodesically convex and its closure at infinity is
precisely $\Sigma$.

The boundary of $D$ is the disjoint union of two weakly spacelike graphs 
$\partial_\pm D=M_{u_\pm}$ whose boundary at infinity is
$\Sigma$.

Every point $p\in \dr D$ is joined to $\Sigma$ by a lightlike ray.
\end{prop}

To prove this proposition we need a technical lemma of AdS geometry.

\begin{lemma}\label{inter:lem}
Given two points $p,q\in AdS_{n+1}$ connected along a geodesic segment
$[p,q]$ and given any point $r$ lying on such a segment, we have that
\[
      U_p\cap U_q\subset U_r\,.
\]
\end{lemma}

\begin{proof}
Let $u_p$ (resp. $v_p$)  be the real function on $\mathbb H^n$ 
such that $P_+(p)$ (resp. $P_-(p)$) is the graph of $u_p$ (resp. $v_p$).
Analogously define $u_q,v_q,u_r,v_r$.

We have that
\[
  U_p=\{(x,t)| v_p(x)<t<u_p(x)\}~,\qquad
  U_q=\{(x,t)|v_q(x)<t<u_q(x)\}~,\qquad U_r=\{(x,t)| v_r(x)<t<u_r(x)\}~.
 \]
 In particular,
$U_p\cap U_q=\{(x,t)| \max\{v_p(x),v_q(x)\}<t<\min\{u_p(x),u_q(x)\}\}$.
Then, the statement turns out to be equivalent to the inequalities
\[
   v_r\leq\max\{v_p,v_q\}\qquad \min\{u_p,u_q\}\leq u_r\,.
\]   

If the segment $[p,q]$ is timelike, then, up to isometry, we can suppose that
$p=(x^0,0)$, $q=(x^0,a)$, $r=(x^0,b)$ with $0\leq b\leq a$.
In this case we have $u_p(x)=\pi/2$, $u_q(x)=a+\pi/2$, $u_r(x)=b+\pi/2$
so the statement easily follows.

Suppose now that the geodesic $[p,q]$ is spacelike.
Up to isometry, we can suppose that  
$p=(x_p,0), q=(x_q,0), r=(x_r,0)$ where
$x_p,x_q,x_r$ are the following points in (the hyperboid model of)
$\mathbb H^n$:
\[
x_p=(-\sh\epsilon,0,\ldots,0,\ch\epsilon),\ 
x_q=(\sh\eta,0,\ldots,0,\ch\eta),\
x_r=(0,\ldots,0,1)\, ,
\]
where $\eta$ and $\epsilon$ are respectively the distance from $p$ and $q$ to $r$.

The corresponding points
$p^*,q^*,r^*\in AdS^*_{n+1}$ are
\[
p^*=(-\sh\epsilon, 0,\ldots,0, \ch\epsilon,0),\quad q^*=(\sh\eta,0,\ldots,0,\ch\eta,0), 
\quad r^*=(0,\ldots,0,1,0)~.
\]
By Remark \ref{planes:rem},
$\Phi(P_+(p))$ is a component of the intersection of $AdS^*_{n+1}$ with the hyperplane
defined by the equation
\[
-y_1\sh\epsilon-y_{n+1}\ch\epsilon=0\,.
\]
In particular, pulling-back this equation, we deduce that the set $P_+(p)$ is a component
of the set
\[
\{((x_1,\ldots,x_{n+1}),t)\in\mathbb H^n\times\mathbb R| -x_1\sh(\epsilon) -
x_{n+1}\cos t\ch(\epsilon)=0\}~.
\]
Since the function $t$ takes value in $(0,\pi)$ on $P_+(p)$ we deduce that
\[
 u_p(x_1,\ldots,x_{n+1})=\mathrm{arccos} 
\left(-\frac{x_1\sh\epsilon}{x_{n+1}\ch\epsilon }\right)~.
 \]  
Analogously, we derive
\[
u_r(x_1,\ldots, x_{n+1})=\pi/2\qquad u_q(x_1,\ldots, x_{n+1})=
\mathrm{arccos} \left(\frac{x_1\sh\eta }{x_{n+1}\ch\eta}\right)~.
\]
Notice that $u_p\leq\pi/2$ if $x_1\leq 0$, whereas $u_q\leq\pi/2$ if $x_1\geq 0$.
It follows that $\min\{u_p,u_q\}\leq u_r$.

Since $v_p=-u_p$, $v_q=-u_q$ and $v_r=-u_r$, we deduce
that $\max\{v_p,v_q\}\geq v_r$.

When $[p,q]$ is lightlike, the computation is completely analogous.
\end{proof}

\begin{remark}\label{ppp:rk}
From the proof of the lemma we have that
$P_+(p)$ and $P_+(q)$ are disjoint in $\overline{AdS}_{n+1}$ if
$p$ and $q$ are joined by a timelike segment, while they
meet along a $(n-1)$-dimensional geodesic plane if $p$ and $q$ are connected by a 
spacelike geodesic. Finally in the lightlike case, they meet
at the asymptotic end-points of the geodesic through $p$ and $q$.
\end{remark}

\begin{proof}[Proof of Proposition \ref{boundarydep:prop}]
Let $p$ be a point contained in $D$ and consider the nearest
conjugate points $p_\pm$ to $p$ as defined in Section \ref{geodesics:sec}.
First we show that $D$ is contained in the star neighbourhood 
$I^-(p_+)\cap I^+(p_-)$ of $p$.
Let $q\notin I^-(p_+)$.
If $q\in \overline{I^+(p_+)}$ then $I^-(p_+)\subset I^-(q)$. Since
$\Sigma$ is contained in the asymptotic boundary of the past of
$P_+(p)=P_-(p_+)$ that in turn coincides with the asymptotic boundary
of $I^-(p_+)$, we see that $\Sigma\subset\partial_\infty I^-(q)$, so that
$\Sigma\cap U_q=\emptyset$.
Suppose now that $q$ is related to $p_+$ by a spacelike geodesic.
Remark \ref{ppp:rk} shows that 
$\partial_\infty P_-(p_+)\cap \partial_\infty P_-(q)$ contains a point $(\xi,t)$.
Since $\Sigma$ is a graph on $\partial\mathbb H^n$, there is a point in $\Sigma$
of the form $(\xi, t')$ and since $\Sigma\subset I^-(P_-(p_+))$ we get $t'<t$.
It follows that $(\xi,t')$ is not contained in $U_q$.
Eventually we obtain that $q\notin D$.  The same argument shows that any point in
$D$ must be contained in $I^+(p_-)$ so  $D$ is contained
in $I^-(p_+)\cap I^+(p_-)$.

We deduce from this that given two points $p,q\in D$, the geodesic segment $[p,q]$ 
joining them exists and does not contain any point conjugate to $p$.
Given a point $r\in[p,q]$ the region $U_r$ contains $U_p\cap U_q$, so 
that $U_r$ contains $\Sigma$. By Lemma \ref{dom:lem} it
follows that $r\in D$. This shows that $D$ is convex.

Clearly $\Sigma$ is contained in the boundary of $D$. On the other
hand, given any other point $q\in\partial_\infty AdS_{n+1}$, the
vertical line through $q$ meets $\Sigma$ at a point $q'$. By Remark
\ref{asym:rk}, there is no geodesic arc in $AdS_{n+1}$ joining $q$ to
$q'$. Since $D$ is convex, $q'$ cannot lie on $D$.  In particular, the
asymptotic boundary of $D$ coincides with $\Sigma$.

To prove that the boundary of $D$ has two components, we notice that 
every timelike geodesic, say $c$, through a point $p\in M$ must intersect 
$\partial D$ in two points which are contained in the future and in the past 
of $M$ respectively. Indeed, since $D$ is contained in some compact region of
$\overline{AdS}_{n+1}$, it turns out that $c\cap D$ is precompact without 
asymptotic points. By the convexity of $D$, we have that $c\cap D$ is a
 compact segment and clearly there is an end-point in the future of $M$ 
 and another end-point in the past of $M$.

Let us define $\partial_\pm D=\partial D\cap I^\pm(M)$.  The previous
argument proves that no timelike geodesic can join points of
$\partial_+D$.  Since $D$ is convex, points of $\partial_+D$ are joined
by lightlike or spacelike geodesic arcs.  In particular $\partial_+D$ is an
acausal set.  By general results (see e.g. \cite{bartnik-simon}) it is a weakly
spacelike surface (in particular it is a $\mathrm C^{0,1}$-embedded
surface).

In addition, every inextensible timelike path without 
endpoints must intersect $\partial_+D$ at some point. 
By  Proposition \ref{spacelikegraphs:prop} we deduce 
that $\partial_+D$ is a   weakly spacelike graph.
 

To conclude we have to prove that points in $\partial D$ are connected
to $\Sigma$ by some lightlike ray.
By the characterization of $D$ given by Lemma \ref{dom:lem}, we have
that $\partial D$ is the set of points $p$ such that $\Sigma\subset
\overline{U_p}$ and $\Sigma\cap \partial_\infty(P_-(p)\cup
P_+(p))\neq\emptyset$.  Take a point $y$ in this intersection.  By the
convexity of $D$, the segment $c$ joining $x$ to $y$ (that is lightlike)
is contained in $\overline{D}$. Points on $c$ are joined to $y\in \Sigma$ 
by a lightlike geodesic, so they cannot be contained in $D$.
In particular $c\subset\partial D$.
\end{proof}



\begin{remark}\label{ct:rem}
Since timelike arcs in $D$ do not contain conjugate points, their length is less than $\pi$.
In particular, the length of  any timelike geodesic segment joining a point of $\partial_-D$
 and a point of $\partial_+D$ is less than $\pi$.
If there exists a point $q_+\in\partial_+D$ and $q_-\in\partial_-D$ such that
$\delta(q_-,q_+)=\pi$, then we have $P_-(q_+)=P_+(q_-)=P$ and 
$\overline{U_{q_+}}\cap\overline{U_{q_-}}=P$. 
Since $\Sigma$ is contained in 
$\overline{U_{q_+}}\cap\overline{U_{q_-}}$, we conclude that
$\Sigma=\partial_\infty P$. 
In this case $D=I^-(q_+)\cap I^+(q_-)$.
\end{remark}



\begin{remark}\label{compact:rem}
The closure of $D$ in $\overline{AdS}_{n+1}$ is compact.
\end{remark}

\begin{lemma}\label{compact:lem}
For every $p\in D$ the intersection
$\overline{I^+(p)}\cap \overline{D}$ is compact in $AdS_{n+1}$.
\end{lemma}

\begin{proof}
Since the closure of $D$ in $\overline{AdS}_{n+1}$ is compact, it is
sufficient to show that no point in $\partial_\infty AdS_{n+1}$ is an
accumulation point for $\overline{D}\cap\overline{I^+(p)}$. However
the set of boundary accumulation points of $\overline{I^+(p)}$ is
disjoint from $U_p$, whereas the set of boundary accumulation points
for $D$ is $\Sigma$, that is contained in $U_p$.
\end{proof}

\begin{lemma}
There is a point $p\in D$ such that $D\subset U_p$.
\end{lemma}

\begin{proof}
We consider first the case there are points $q_+\in\partial_+D$ and
$q_-\in\partial_-D$ such that $\delta(q_-,q_+)=\pi$.
By Remark \ref{ct:rem}, we deduce that $D=I^-(q_+)\cap I^+(q_-)$ and 
any point on the plane $P_-(q_+)=P_+(q_-)$ satisfies the statement.

Now we consider the case where $\delta(q,q')<\pi$
for $q\in\partial_-D$ and $q'\in\partial_+D$.
We define two functions on $D$
\[
\tau_+(p)=\sup_{q\in D\cap I^+(p)}\delta(p,q)\qquad
\tau_-(p)=\sup_{q\in D\cap I^-(p)}\delta(q,p)
\]
that are Lipschitz-continuous (see \cite{benedetti-bonsante}).
By Lemma \ref{compact:lem}, for $p\in D$ there is $q_+(p)\in \overline{D}$ such that
$\tau_+(p)=\delta(p, q_+(p))$ and analogously there is a point $q_-(p)$ such that 
$\tau_-(p)=\delta(q_-(p),p)$. Clearly $q_+(p)\in\partial_+D$ and $q_-(p)\in\partial_-D$.

Notice that by the reverse of triangle inequality we have 
$\tau_+(p)+\tau_-(p)\leq\delta(q_-(p),q_+(p))<\pi$.
In particular the open sets $\Omega_-=\{\tau_-<\pi/2\}$ and $\Omega_+=\{\tau_+<\pi/2\}$ cover
$D$. Since they are not empty, it follows that there exists a point $p$ such that
$\tau_-(p)<\pi/2$ and $\tau_+(p)<\pi/2$, so $D\subset U_p$.
\end{proof}

\subsection{From space-like graphs in $AdS_3$ to diffeomorphisms of $\HH^2$}
\label{ssc:diffeos}

There is a relation between some space-like surfaces in $AdS_3$ 
(satisfying some specific properties) and diffeomorphisms from $\HH^2$ to
$\HH^2$. More specifically, there is a one-to-one relation between maximal
graphs in $AdS_3$ with negative sectional curvature and minimal Lagrangian
diffeomorphisms from the hyperbolic disk to itself. The quasi-conformal
minimal Lagrangian diffeomorphisms correspond precisely to the maximal graphs
with uniformly negative sectional curvature.

This relation, which is well-known (see \cite{AAW}), is at the heart of the proof
of Theorem \ref{tm:ml}, so we outline its construction and its main properties here,
refering to \cite{mess,mess-notes,minsurf,cone,collision} for more details.

Let $S\subset AdS_3$ be a space-like graph. Let $I$ be its induced metric,
$B$ its shape (or Weingarten) operator, and let $E$ be the identity map from
$TS$ to $TS$ at each point. Denote by $J$ the complex structure of $I$
on $S$. We can then define two metrics $\mu_l,\mu_r$ as :
$$ \mu_l=I((E+JB)\cdot, (E+JB)\cdot)~, ~~ \mu_r = I((E-JB)\cdot, (E-JB)\cdot)~. $$
It is then not difficult to show that both $\mu_l$ and $\mu_r$ are hyperbolic
metrics (see \cite{minsurf,collision}) -- the reason for this being that 
$E\pm JB$ satisfies the Codazzi equation, $d^\nabla (E\pm JB)=0$ on $S$, and
that $\det(E\pm JB)=1+\det(B)$ is equal to minus the sectional curvature of
the induced metric $I$ on $S$, which by the Gauss equation in $AdS_3$ 
is equal to $-1-\det(B)$.

However $\mu_l$ and $\mu_r$ are not necessarily smooth metrics, they might
have singularities when $E\pm JB$ is singular, that is -- by the determinant
computation just mentioned -- when $1+\det(B)=0$. This means that $\mu_l$
and $\mu_r$ are smooth hyperbolic metrics whenever the induced metric on
$S$ has negative sectional curvature.


There is a nice geometric interpretation of metrics $\mu_l$ and $\mu_r$
that is based on a specific feature of $AdS_3$.

Every leaf of the left (right) foliation of $\dr_\infty AdS_3$ meets the boundary
of any spacelike planes exactly at one point.
Consider a fixed totally geodesic plane $P_0$.
Given any other plane $P$ there are two natural identifications
$\Phi_{P,l},\Phi_{P,r}:\dr_\infty P\rightarrow\dr_\infty  P_0$ obtained
by following each of the families of lines $\cL_l,\cL_r$.

By means of the projective model, it can be easily seen that maps
$\Phi_{P,l}$ and $\Phi_{P,r}$ extend uniquely to isometries of $AdS_3$
-- still denoted by $\Phi_{P,l}$, $\Phi_{P,r}$ -- sending $P$ to $P_0$
(see \cite{mess, mbh} for details).

It is also not difficult to check that replacing $P_0$ by another geodesic 
plane does not change $\Phi_{P,l}$ and $\Phi_{P,r}$ up to left composition 
by some isometry of $AdS_3$ preserving respectively $\cL_l$ and $\cL_r$.

Now given any spacelike surface $S$ we can define two maps
$\Phi_l,\Phi_r:S\rightarrow P_0$ as
\[
  \Phi_l(x)=\Phi_{P(x),l}(x)\qquad\Phi_r(x)=\Phi_{P(x),r}(x)\,,
\]
where $P(x)$ is the geodesic plane tangent to $S$ at $x$.
Still in this case, replacing $P_0$ does not change $\Phi_l$ and $\Phi_r$,
up to left composition with some isometry of $AdS_3$ that preserves
respectively $\cL_l$ and $\cL_r$.

The following is a basic remark, see e.g. \cite{minsurf} for a proof --
it can actually be checked by a 
direct computation, by choosing $P_0$ as the tangent plane at the point $x$.

\begin{lemma} \label{lm:local}
The pull-backs by
$\Phi_l$ (resp. $\Phi_r$) of the hyperbolic metric on $P_0$ is precisely
the metric $\mu_l$ (resp. $\mu_r$). 
\end{lemma}

A consequence is that $\Phi_l$ and $\Phi_r$ are non-singular when $\mu_l,\mu_r$
are non-degenerate metrics, and we have seen that this is the case when
$\det(B)\neq -1$. We are therefore lead to consider surfaces with negative
sectional curvature (the Gauss formula indicates that the sectional curvature
of $S$ is $K=-1-\det(B)$).

Lemma \ref{lm:local}, which is a local statement, can be improved, 
under the condition that $S$ is a 
space-like maximal graph with negative curvature.
Here we call $\pi_l$ (resp. $\pi_r$) the map from $\dr_\infty AdS_3$ to 
$P_0$ sending a point $x\in \dr_\infty AdS_3$ to the intersection with $P_0$
of the line of $\cL_l$ (resp. $\cL_r$) containing $x$.

\begin{prop} \label{prop:global}
Suppose that $S$ is a maximal space-like graph with 
sectional curvature bounded from above by some negative constant.
Then $\Phi_l$ (resp. $\Phi_r$) is a global diffeomorphism from $S$ to $P_0$. 
$\Phi_l$ (resp. $\Phi_r$) extends continuously to the closure of $S$ in 
$\overline{AdS_3}$, and its boundary value is the restriction of $\pi_l$ 
(resp. $\pi_r$) to $\dr_\infty S$.
\end{prop}

The difficult part to prove  is the extension result.
We need the following technical lemma that gives a condition
for the extension. Unfortunately this lemma does not apply directly to 
$S$, but to the surface $S^+$ of points whose distance from $S$ is $\pi/4$.
We then factorize the map $\Phi_l$  as the composition of the corresponding map
$\Phi_l^+:S^+\rightarrow P_0$ and a diffeomorphism $\sigma:S\rightarrow S_+$
that is given by the normal evolution and  that is the identity on the boundary.

\begin{lemma}\label{crit:lem}
Let $S$ be a spacelike surface in $AdS_3$ with negative curvature
whose boundary curve $\Gamma$ 
does  not contain singular points (that is,
$\partial_\infty S$ does not contain any lightlike segment). Consider
the maps $\Phi_l,\Phi_r:S\rightarrow P_0$ described above.  Suppose that
there is no sequence of points $x_n$ on $S$ such that
the totally geodesic  planes $P_n$ tangent to $S$ at $x_n$ converge
to a lightlike plane $P$ whose past end-point and future end-point
are not on $\Gamma$.

Then for any sequence of points $x_n\in S$ converging to
$x\in\partial_\infty S$ we have that $\Phi_l(x_n)\rightarrow \pi_l(x)$
(resp. $\Phi_r(x_n)\rightarrow\pi_r(x)$) in $\overline{P_0}$
\end{lemma}

\begin{proof}
We prove that for any sequence
$x_n\rightarrow x\in \partial_\infty S$ there is a subsequence such that
$\Phi_l(x_{n_k})$ converges to
$\pi_l(x)$.  

Indeed, up to passing to a subsequence we can suppose that the
totally geodesic plane $P_n$ tangent to $S$ at $x_n$ converges to a
plane $P_\infty$.  Since $x$ is the limit of points on $P_n$, it
belongs to $\partial_\infty P_\infty$.

We distinguish two cases
\begin{enumerate}
\item $P_\infty$ is spacelike;
\item $P_\infty$ is lightlike.
\end{enumerate}

First we deal with the first case.
We have that $\Phi_l(x_n)=\Phi_{P_n,l}(x_n)$. Since $P_n\rightarrow
P_\infty$ it can be checked that $\Phi_{P_n,l}\rightarrow\Phi_{P,l}$ uniformly
on $\overline{AdS}_3$ (see \cite{mbh}). So we have
\[
   \Phi_l(x_n)\rightarrow\Phi_{P_\infty,l}(x)=\pi_l(x)\,.
\]

\begin{figure}
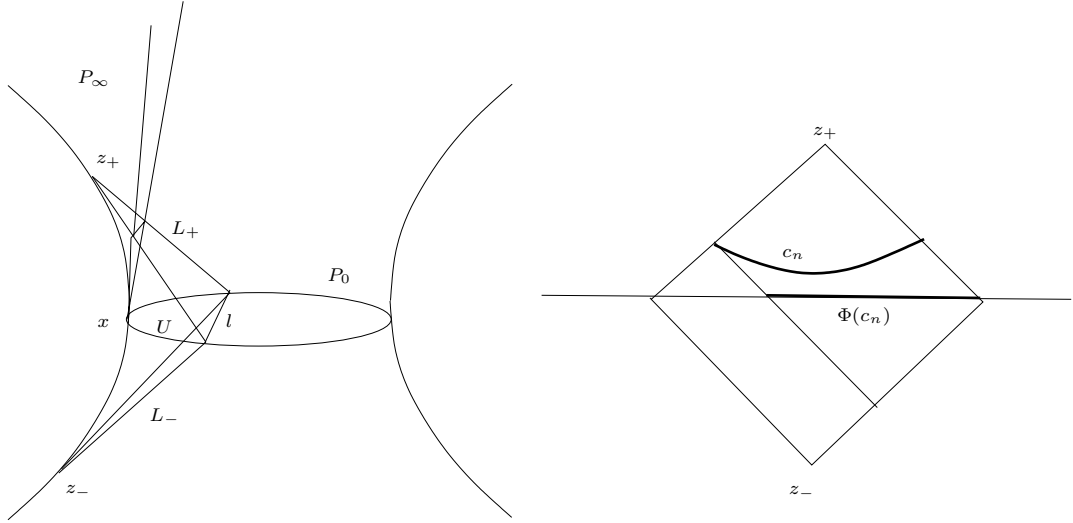

\begin{center}
\input rombus.pstex_t
\end{center}
\caption{The rhombus in the proof of Lemma \ref{crit:lem}}\label{fg:rhombus}
\end{figure}

Consider now the case where $P_\infty$ is lightlike.
By the assumption either the past or the future end-point of $P_\infty$
is contained in $\Gamma=\dr_\infty S$. Since points on $\Gamma$ are not joined
by lightlike segments, the intersection between $\Gamma$ and
$P_\infty$ is only this point. Since $x\in\Gamma\cap P_\infty$,
we conclude that $x$ is either the past endpoint or the future end-point 
of $P$. Up to reversing the time-orientation we can suppose
that $x$ is the past end-point of $P_\infty$.

Up to some isometry of $AdS_3$ preserving the leaves of $\cL_l$ we can suppose
that $x\in P_0$ so it is sufficient to prove that $\Phi_l(x_n)\rightarrow x$.

Consider any geodesic $l$ on $P_0$ and let $U$ be the half-plane
bounded by $l$ containing the point $x$. We will show that
for $n$ large enough $\Phi_l(x_n)\in U$.

The four leaves of $\cL_l$ and $\cL_r$
passing through the end-points of $l$ 
bound a rhombus $R$
in $\partial_\infty AdS_3$ containing $x$ in its interior (see Figure \ref{fg:rhombus}). 
The end-points of $l$ are two opposite vertices of $R$ and 
there are  two other opposite vertices $z_-$ and $z_+$
such that  $z_-$ is the past end-point of both edges
adjacent to it and $z_+$ is the future end-point of both edges
adjacent to it.

Since $x$ is the past endpoint of $P_\infty$,
this plane intersects the frontier of $R$ in two points,
one for each edge with vertex $z_+$. In particular also
$P_n\cap R$ is for $n$ large enough an arc  $c_n$ joining two points
on the edges adjacent to $z_+$.

Let $L_-$ be the lightlike plane whose past end-point is $z_-$ and
$L_ +$ be the lightlike plane whose future end-point is $z_+$.
Notice that $V=I^-(L_+)\cap I^+(L_-)$ is a neighbourhood
of $x$ in $\overline{AdS}_3$ and the asymptotic boundary of $V$ is exactly
$R$. In particular, for $n$ large enough, $x_n\in V$.

The boundary of $L_+$ is the union of the two past-directed lightlike
rays starting from $z_+$ and $L_-$ is the union of two future-directed lightlike
rays starting from $z_+$.

It turns out that  $H_n=P_n\cap I^-(L_+)$ is the half-plane on $P_n$ 
that is the convex hull of $c_n$. Since $c_n$ is contained in  
the future of $\partial_\infty L_-$ we  have that $H_n\subset I^+(L_-)$.
And we conclude that
\[
  P_n\cap V=H_n\,.
\]
Since for $n$ large enough $x_n\in P_n\cap V$, we have that
\[
   \Phi_l(x_n)=\Phi_{P_n,l}(x_n)\in \Phi_{P_n,l}(H_n)\,.
\]

Now $\Phi_{P_n,l}(H_n)$ is the half-plane of $P_0$ whose
asymptotic boundary is $\pi_l(c_n)$.

Notice that $\pi_l(c_n)$ is contained in $\partial_\infty U$
so we have $\Phi_l(x_n)\in\Phi_{P_n,l}(H_n)\subset U$.
\end{proof}

\begin{remark}\label{bo:rk}
If $S$ is a future-convex graph and its boundary does not contain
singular points then the condition required in Lemma \ref{crit:lem} is
satisfied.  Indeed totally geodesic planes tangent to $S$ are support
planes so if we take a sequence of such planes $P_n$ that converges
to some lightlike plane $P_\infty$, we have that $P_\infty$ cannot
intersects $S$ transversally.  In particular $S$ is contained in the
past of $P_\infty$.  This implies that either the boundary of $S$ is
disjoint from the boundary of $P_\infty$ or that the past end-point of
$P_\infty$ is contained in the boundary of $S$.

Now if the tangency points $x_n$ of $P_n$ with $S$ converge to
some asymptotic point $x$, clearly $x\in S\cap P_\infty$.
Thus, in this case we have that the past end-point
of $P_\infty$ is contained in the boundary of $S$. Since the boundary of
$S$ does not contain lightlike segments, the point $x$ must
coincide with the past end-point of $P_\infty$.
\end{remark}

\begin{lemma} \label{lm:320}
Let $S$ be a maximal spacelike graph with sectional curvature bounded from above
by some negative constant.  The asymptotic boundary of $S$ does not contain any lightlike
segment.
\end{lemma}

The proof is based on some simple preliminary claims.

\begin{claim} \label{cl:equidistant}
Let $S\subset AdS_3$ be a space-like graph with principal curvatures in $(-1,1)$.
Then the equidistant surfaces $S_r$ at (oriented) time-like distance $r$ from
$S$, for all $r\in (-\pi/4,\pi/4)$, are smooth, space-like graphs. If the
principal curvatures of $S$ are in $(-1+\epsilon,1-\epsilon)$, then, for 
$r$ close enough to $\pi/4$, $S_{r}$ 
is past-convex, and $S_{-r}$ is future-convex. 
\end{claim}

\begin{proof}
If $(S_r)_{r\in I}$ is a 
non-singular foliation of a neighborhood of $S$ by space-like surfaces at constant distance 
$r$ from $S$, then the shape operator $B_r$ of $S_r$ satisfies a Riccatti type
equation relative to $r$:
$$ \frac{dB_r}{dr} = B_r^2-I~, $$
where $I$ is the identity. It follows that the principal curvatures of $S$
evolve as $\tan(r-r_0)$, where $r_0$ is chosen so that $\tan(r_0)$ is the
principal curvature of $S$ at the corresponding point and in the corresponding
direction.

Suppose now that $S$ has principal curvatures 
$k\in (-1+\epsilon,1-\epsilon)$ at each point, for some $\epsilon>0$.
This implies that, at each point and in
each principal direction, $r_0\in (-\pi/4+\alpha, \pi/4-\alpha)$, where
$\alpha>0$ is another constant. As a consequence, the
equidistant foliation $(S_r)$ is well-defined for $r\in [-\pi/4,\pi/4]$, 
and moreover the surfaces $S_{\pi/4-\alpha}$ and $S_{-\pi/4+\alpha}$ are smooth and respectively
strictly concave and strictly convex, so that the domain 
$$ \Omega=\cup_{r\in [-\pi/4+\alpha,\pi/4-\alpha]}S_r $$ 
is convex with smooth boundary, with principal curvatures bounded from below by a 
strictly positive constant. 
\end{proof}

\begin{cor} \label{cr:w}
Let $S$ be a space-like maximal surface, with sectional curvature bounded from
above by a negative constant. Then $w(S)<\pi/2$.
\end{cor}

\begin{proof}
This follows from the claim because the convex hull of $S$ is contained in 
$\Omega$, and $w(\Omega)\leq \pi/2-2\alpha<\pi/2$.  
\end{proof}

\begin{claim} \label{cl:light}
Suppose that there is a light-like segment in $\dr_\infty S$. Then $w(S)=\pi/2$.
\end{claim}

\begin{figure}[h!]
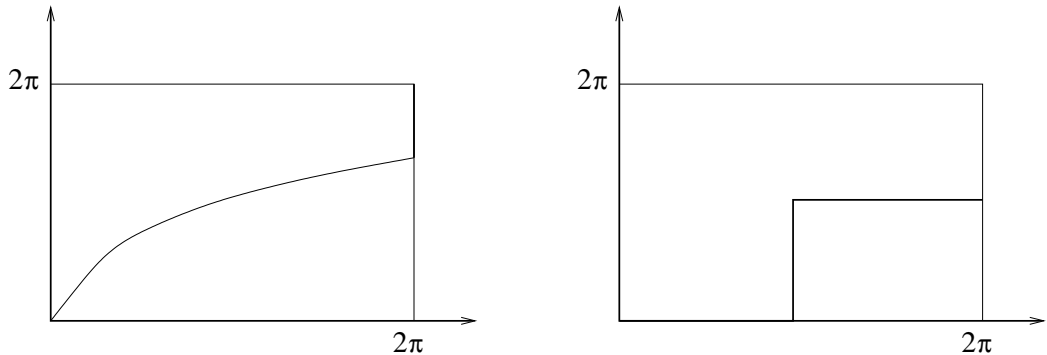
 
\begin{center}
\input graphs.pstex_t
\end{center}
\caption{Deforming a graph to the standard 2-step graph}\label{fg:graphs}
\end{figure}

\begin{proof}
The boundary at infinity of $S$ is the graph of a map $u:S^1\rightarrow S^1$.
If $\dr_\infty S$ contains a light-like segment then $u$ is not continuous, and
its graph has a ``jump'', as in the left-hand side of Figure \ref{fg:graphs}.
Composing $u$ on the left with a sequence of projective transformations, we can
make its graph as close as wanted (in the Hausdorff topology) from the standard
2-step graph shown on the right-hand side of Figure \ref{fg:graphs}. (This is 
achieved by composing $u$ on the right with a sequence of powers of a projective
transformation having as attracting fixed point the point where the ``jump'' occurs.) 
We call $\Gamma_0$ this 2-step graph, considered as a subset of $\dr \pi(AdS_3)$ (here
$\pi$ is the map in the projective model of $AdS_3$).

Now $\Gamma_0$, as a subset of $\dr \pi(AdS_3)$, is composed of four light-like
segments. It has four vertices, and it is not difficult to check that the 
lines $\Delta$ and $\Delta^*$ connecting the two pairs of opposite points 
are two dual space-like lines in $\pi(AdS_3)$. In particular, if $CH(\Gamma_0)$
denotes the convex hull of $\Gamma_0$, then $w(CH(\Gamma_0))=\pi/2$.

Since $\dr_\infty S$ can be made arbitrarily close to $\Gamma_0$ by applying
AdS isometries (corresponding to composing $u$ on the left and on the right
with projective transformations of $S^1$), it follows that $w(S)=\pi/2$.
\end{proof}

\begin{proof}[Proof of Lemma \ref{lm:320}]
The statement follows directly from Corollary \ref{cr:w} and Claim \ref{cl:light}.
\end{proof}

Let us come back to Proposition \ref{prop:global}.

\begin{proof}[Proof of Proposition \ref{prop:global}]
We consider again the surface $S_+$  of points in the future of $S$ 
at distance $\pi/4$ from $S$.
We have seen that $S_+$ is smooth and past-convex.
Moreover  a diffeomorphism $\sigma:S\rightarrow S_+$ is uniquely determined 
so  that 
the Lorentzian distance between $x$ and $\sigma(x)$ is exactly $\pi/4$.

Since the distance between points on $S_+$ and points on $S$ is
bounded, they share the same boundary. 
Moreover, since the boundary of $S$ does not contain lightlike segments, it can easily
seen that the map $\sigma$ extends to the identity at the boundary.

We claim that the map $\Phi_l$ can be factorized as the composition
of $\sigma$ and $\Phi^+_l$, where $\Phi^+_l:S_+\rightarrow P_0$ is the map
constructed in the same way as $\Phi_l$.
The claim and  Remark \ref{bo:rk} imply that $\Phi_l$ extends to the boundary.

Let us prove the claim. Given any point $x\in S$, we have to
check that $\Phi_l(x)=\Phi_l^+(\sigma(x))$.
Up to isometry we can suppose that:
\begin{itemize}
\item $P_0$ is the  plane tangent to $S$ at $x$,
\item $x=(x^0,0)$ and $P_0$ is the horizontal plane.
\end{itemize}

With this assumption clearly $\Phi_l(x)=x$.

Since the segment joining $x$ to $\sigma(x)$ is orthogonal to both
$S$ and $S^+$, it follows that $\sigma(x)=(x^0,\pi/4)$ 
and the plane $P_+$ tangent to $S_+$ at $\sigma(x)$ is the horizontal plane.

In this case the map $\Phi_{P_+,l}$ can be explictly computed. In particular it is given by 
$\Phi_{P_+,l}(y,t)=(R(y),t-\pi/4)$ where $R\in Isom(\mathbb H^2)$ 
is a rotation of angle $\pi/4$ around $x^0$.
It easily follows that $\Phi_l^+(\sigma(x))=\Phi_{P_+,l}(\sigma(x))=x$, and this proves
the claim.

Notice that the map $\Phi_l$ and $\Phi_r$ turn to be proper maps.
On the other hand, under the hypothesis that $S$ has negative sectional curvature, 
$\Phi_l$ and $\Phi_r$ are local diffeomorphisms from $S$ to $P_0$, 
so that, by the Dependence of Domain Theorem, they are global
diffeomorphism from $S$ to $P_0$.
\end{proof}

\begin{defi}
Suppose that $S$ has negative sectional curvature. We call 
$\Phi_S:\Phi_l^{-1}\circ \Phi_r:\HH^2\rightarrow \HH^2$. 
$\Phi_S$ is a global diffeomorphism, well-defined up
to composition by a hyperbolic isometry.
\end{defi}

By construction the differential of $\phi_S$ is given at each point by 
$(E+JB)^{-1}(E-JB)$. It follows that, as long as the principal 
curvatures of $S$ are in $[-1+\epsilon, 1-\epsilon]$ for some
$\epsilon>0$, the diffeomorphism $\phi_S$ is quasi-conformal (and conversely).

\begin{lemma}
The map $\Phi_S$ extends to a homeomorphism from $\overline{\HH^2}$ to 
$\overline{\HH^2}$, and the graph of $\dr\Phi_S:S^1\rightarrow S^1$ 
in (the image by $\pi$) of $AdS_3$ is the boundary at infinity of $S$ 
in $\dr_\infty AdS_3$.
\end{lemma}

\begin{proof}
The extension of $\Phi_S$ to the boundary is a direct consequence of 
its definition and of the extension to the boundary of $\Phi_l$ and
$\Phi_r$. It is then clear that the graph of $\dr \Phi_S$ is equal
to $\dr_\infty S$, since the restrictions of $\pi_l$ and $\pi_r$ to
$\dr_\infty S$ are equal to the boundary values of $\Phi_l$ and $\Phi_r$.
\end{proof}

We have now proved the first two points in Proposition \ref{pr:translation}.
To prove the third point it is necessary to construct, given a quasi-conformal minimal
Lagrangian diffeomorphism $\Phi:\HH^2\rightarrow \HH^2$, a maximal space-like 
$S$ such that $\Phi=\Phi_S$. One way to do this is through the identification
of $\HH^2\times \HH^2$ with the space of time-like geodesics in $AdS_3$ (see
\cite{collision}). We rather use here local arguments (as in \cite{minsurf}).

Let $\Phi:\HH^2\rightarrow \HH^2$ be a minimal Lagrangian diffeomorphism. 
Call $\rho_l$ and $\rho_r$ the hyperbolic metrics on the two copies of 
$\HH^2$ (this underlines the relationship with the construction in the
previous paragraphs). The fact that $\Phi$ is minimal Lagrangian is
equivalent (see \cite{L5}) to the fact that 
$$ \Phi^*\rho_r = \rho_l(b\cdot, b\cdot)~, $$
where $b$ is self-adjoint (for $\rho_l$), of determinant $1$, and satisfies
the equation 
$$ d^{\nabla^l} b=0~, $$
where $\nabla^l$ is the Levi-Civita connection of $\rho_l$ and $d^{\nabla^l} b$ is
defined (see \cite{Be}) as
$$ (d^{\nabla^l} b)(x,y)= \nabla^l_x(by)-\nabla^l_y(bx)-b([x,y])~. $$
We can then define a metric $I$ on $S$ by
\begin{equation}
  \label{eq:defI}
  4I = \rho_l((E+b)\cdot,(E+b)\cdot)~. 
\end{equation}
Since $b$ is non-singular and has positive eigenvalues, $I$ is a metric on $\HH^2$.
Since $d^{\nabla^l}b=0$ we also have $d^{\nabla^l}(E+b)=0$, it follows
from standard arguments (see e.g. \cite{minsurf}) that the Levi-Civita connection
of $I$ is 
$$ \nabla_xy = (E+b)^{-1}\nabla^l_x((E+b)y)~, $$
and therefore that the curvature $K$ of $I$ is equal to 
$$ K = \frac{K_l}{\det((E+b)/2)} = -\frac{4}{\det(E+b)} =
-\frac{4}{2+\tr(b)}~. $$
Let $J$ be the complex structure of $I$, we now define $B:T\HH^2\rightarrow T\HH^2$ 
as follows:
\begin{equation}
  \label{eq:defJB}
   JB = (E+b)^{-1}(E-b)~. 
\end{equation}
Then $JB$ has some remarkable properties.
\begin{enumerate}
\item $d^\nabla JB=0$. This follows from a direct computation, because $d^{\nabla^l}(E-b)=0$. 
Since $J$ is parallel for $\nabla$, it follows that $d^\nabla B=0$. 
\item $JB$ is self-adjoint for $I$, because $E-b$ is self-adjoint for $\rho_l$. 
It follows that $B$ is traceless. 
\item $JB$ is traceless -- this follows from a direct computation in a basis where
$b$ is diagonal, using the fact that $\det(b)=1$. It follows that $B$ is self-adjoint.
\item $\det(JB)=\frac{\det(E-b)}{\det(E+b)}=\frac{2-\tr(b)}{2+\tr(b)}$. It follows 
that $K=-1-\det(B)$. 
\end{enumerate}
In other terms, setting $\II=I(B\cdot,\cdot)$, we see that $\II$ satisfies the
Gauss and Codazzi equation relative to $I$. It follows that there exists
a (unique) isometric embedding of $(\HH^2, I)$ in $AdS_3$ with second fundamental
form $\II$ (and shape operator $B$). 

Equation (\ref{eq:defJB}) then shows that $E+JB=2(2+b)^{-1}$, so that $\mu_l=\rho_l$,
and a direct computation shows also that $\mu_r=\rho_r$. If $\Phi$ is quasi-conformal
then $b$ is bounded, so that the sectional curvature of $S$ is uniformy negative.
The first part of this section shows that the graph of $\dr\Phi$ in $S^1\times S^1
\simeq \dr_\infty AdS_3$ is equal to the boundary at infinity of $S$, and this 
finishes the proof of Proposition \ref{pr:translation}.

\section{The existence and regularity of maximal graphs}
\label{sc:maximal}

Given a   smooth
spacelike surface $M$  in $AdS_{n+1}$ we consider the
future-oriented normal vector field $\nu$.

The gradient function with respect to the field  $T=-\phi\bar\nabla t$ is
\[
    v_M=-\langle\nu, T\rangle\,.
\]
It measures the angle between the hypersurface $M$ and the horizontal slice.
Notice that $v_M(x)\geq 1$ for every $x\in M$.
If $M$ is the graph of a function $u$ then 
\[
   v_M=\frac{1}{\sqrt{1-\phi^2|\bar\nabla u|^2}}\,.
\]
In that case the normal field $\nu$ is equal to $\nu=\phi v_M(\nabla u-\nabla t)$.

The shape operator of $M$ is the linear operator of $TM$ defined by
\[
   B(v)=\bar\nabla_v\nu
\]
whereas the second fundamental form is defined by
$\II(v,w)=\langle v, B(w)\rangle$.
The mean curvature, denoted by $H$, is the trace of $B$.

In \cite{bartnik-existence} a general formula
for the mean curvature of a spacelike graph is given.
If $M$ is the spacelike graph of a function $u$ we have
\begin{equation}\label{meancurvature:eq}
H=\frac{1}{v_M}\left(\mathrm{div}_M(\phi \mathrm{grad}_M u)+\mathrm{div}_MT\right)~,
\end{equation}
where $\mathrm{div}_M$ is the operator on $M$ defined 
\[
   \mathrm{div}_M X=\sum \langle e_i,\bar\nabla_{e_i}X\rangle,\qquad
   X\in\Gamma(TAdS_{n+1})
\]
where $e_i$ is any orthonormal basis.


A spacelike surface $M$ is \emph{maximal} if its mean curvature vanishes.


\subsection{Maximal hypersurfaces and convex subsets}

We concentrate here on convexity properties of maximal hypersurfaces in $AdS_{n+1}$.

\begin{lemma}\label{planes:lem}
Let $M$ be a compact maximal graph.
Suppose that there exists a spacelike plane $P$ such that
$\partial M$ is contained in $I^-(P)$. 
Then $M$ is contained in $I^-(P)$.
\end{lemma}

\begin{proof}
Suppose by contradiction that a point $p_0$ of $M$ lies in the future of $P$.
Without loss of generality we can suppose that $P$ is the horizontal
plane $\{t=0\}$ and $p_0=(x^0,a)$ with $a>0$.  
Since $M$ is contained in $I^+((p_0)_-)\cap I^-((p_0)_+)$, by our assumption on the boundary
we have that $0<a<\pi$ and $\partial M$ is contained in the region
of points with $-\pi<t<0$.

Consider the function $u:AdS_{n+1}\rightarrow\R$
defined at the point $p=(x,t)$ as
\[
  u(p)=x_{n+1}\sin(t)~.
\]

By our assumption,
\begin{equation}\label{neg:eq1}
u(p)< 0\quad \textrm{for  every }p\in\partial M\,.
\end{equation}

We compute now $\Delta u$, where $\Delta$ is the Beltrami-Laplace operator of $M$.
Notice that $u$ is the pull-back of the function $u^*$ defined on $AdS^*_{n+1}$ as
\[
  u^*(y)=\langle y, e\rangle~,
\]
where $e=(0,\ldots,0,-1)$. Thus we can suppose that $M$ is immersed in
$AdS^*_{n+1}$ and compute $\Delta u^*$. Notice that the gradient of
$u^*$ is the orthogonal projection of $e$ on $M$, that is,
\[
\nabla u(y)= e + \langle e,y\rangle y +\langle e, \nu^*\rangle \nu^*= 
e + uy+\langle e,\nu^*\rangle\nu^*~,
\]
where $\nu^*$ is the normal field of $M$ in $AdS^*_{n+1}$.
Since for $v\in T_yM$, $\nabla_v(\nabla u)$ is the tangential part
of $\bar\nabla_v(\nabla u)$ (where $\bar\nabla$ is the standard connection
in $\mathbb R^{2,2}$) we have 
\[
\nabla_v(\nabla u^*)=u^*v +\langle e,\nu^*\rangle B(v)~.
\]
Taking the trace we get $\Delta u^*=nu^* +\langle e,\nu^*\rangle H=nu^*$, where
the last equality holds since $M$ is maximal.
Eventually we have
\[
  \Delta u= nu\,.
\]

In particular if the maximum of the function $u$ is achieved at some interior point of
$M$, then it must be negative. Since $u(p_0)>0$ we get a contradiction.
\end{proof}

\begin{defi}
A convex slub of $AdS_{n+1}$ is  a convex domain in $AdS_{n+1}$ whose boundary is the union
of two  acausal graphs.
\end{defi}
Let $K$ be a convex slub and $M_v$ and $M_u$ be its boundary components with
$v<u$. The domain $K$ is
\[
\{ (x,t)| u(x)\leq t\leq v(x)\}~.
\]
The component $M_v$ (resp. $M_u$) is called the past (resp. future) boundary of $K$.
Notice that the future boundary is past-convex: this means that
points of $M_v$ are related by a spacelike geodesic that lies in the past  of $M_v$.  
Analogously $M_u$ is future convex.

Since points of a convex slub $K$ are connectible by geodesics, Remark
\ref{asym:rk} implies that the asymptotic boundary of $K$ can
intersect each vertical line in $\partial_\infty AdS_{n+1}$ in at most
one point. So we have

\begin{cor}\label{bound:cor}
If $K$ is a convex slub then its boundary components share the same asymptotic boundary.
\end{cor}

\begin{remark}
Let $u$  and $v$ be two spacelike functions defined on $\mathbb H^n$ such that
$M_u$ is past convex, $M_v$ is future convex and $v(x)<u(x)$.
Corollary \ref{bound:cor} implies that  in general  the domain 
$\Omega=\{(x,t)| v(x)<t<u(x)\}$ is not convex.
On the other hand it is not difficult to see that
if the functions $u$ and $v$ coincide on $\partial\mathbb H^n$, then
$\Omega$ is a convex slub.
\end{remark}

\begin{remark}
Let $K$ be a convex slub and $D$ be the domain of dependence of its asymptotic boundary.
Then $K$ is contained in $\bar D$.
\end{remark}

An important property of convex slubs is that a maximal surface whose boundary 
is contained in a convex slub is completely contained in the slub.

\begin{prop}\label{convexity:prop}
Let $\Omega$ be a convex slub. If $M$ is a compact maximal surface such that
$\partial M$ is contained in $\Omega$. Then $M$ is contained in $\Omega$. 
\end{prop}

Proposition \ref{convexity:prop} is a direct consequence of Lemma
\ref{planes:lem} and the following lemma.

\begin{lemma}\label{lm:omega}
Let $\Omega$ be a convex slub and let $S_-,S_+$ denote respectively its past and future boundary.
For every $p\in S_-$ (resp. $p\in S_+$) there is a spacelike geodesic plane $P_p$ passing through
$p$ such that $\Omega\subset I^+(P_p)$ (resp. $\Omega\subset I^-(P_p)$).

Moreover we have
\[
 \Omega=\bigcap_{p\in S_-}I^+(P_p)\cap \bigcap_{p\in S_+}I^-(P_p)~.
\]
\end{lemma}

\begin{proof}
Since $\Omega$ is contained in the domain of dependence $D$ of its asymptotic boundary, 
there is a point $p$ such that $\Omega\subset U_p$.
Up to isometry we can suppose that $p=(x^0,0)$ and 
 consider the projective map
\[
  \pi^*: U_p\rightarrow \mathbb R^{n+1}
\]
constructed in Section \ref{proj:sec}.
Since $\pi^*$ is a projective map, the set $\pi^*(\Omega)$ is convex in $\mathbb R^{n+1}$.
 
Given a point $q\in S_+$ the point $q^*=\pi^*(q)$ lies on the boundary
of $\pi^*(\Omega)$, so there is a support plane $P^*$ passing through
it.  We can consider the plane in $U_p$ equal to
$P_q=(\pi^*)^{-1}(P^*)$.  This plane passes through $q$ and does not
meet the interior of $\Omega$.  Since any timelike arc passing through
$q$ meet the interior of $\Omega$, the plane $P_q$ is not timelike.
In particular $P$ disconnects $AdS_{n+1}$ in two components that are
the future and the past of $P_q$. Since $q\in S_+$ it turns out that
$\Omega\subset I^-(P_q)$.  Analogously for $q\in S_-$ we find a plane
$P_q$ such that $\Omega\subset I^+(P_q)$.

In particular the inclusion
\[
 \Omega\subset\bigcap_{p\in S_-}I^+(P_p)\cap \bigcap_{p\in S_+}I^-(P_p)
\]
is proved.
Now take a point $q\notin\Omega$. Consider a timelike geodesic arc 
contained in $AdS_{n+1}\setminus\Omega$ such that $q$ is an end-point 
and the other end-point, say $p$, lies on $\partial \Omega$.
Without loss of generality we can assume $p\in S_+$.
In that case it turns out that $q\in I^+(P_p)$, so the reverse inclusion is also proved.
\end{proof}

\begin{lemma} \label{lm:ch}
Let $\Sigma$ be a spacelike graph in $\partial_\infty AdS_{n+1}$.
There is a convex slub $K(\Sigma)$, called the convex hull of $\Sigma$, 
such that :
\begin{itemize}
\item The asymptotic boundary of $K(\Sigma)$ is $\Sigma$.
\item Every convex slub with boundary $\Sigma$ contains $K(\Sigma)$.
\end{itemize}
\end{lemma}

\begin{proof}

Let $D$ be the domain of dependence of $\Sigma$ and take $p\in D$.

Consider the image $\Sigma^*$ of $\Sigma$ through the projective map
\[
   \pi^*: U_p\rightarrow\mathbb R^{n+1}~.
\]
Clearly $\Sigma^*$ is contained in the image, say $D^*$, of $D$.
In particular the convex hull in $\mathbb R^{n+1}$ of $\Sigma^*$, say $K$, is contained in $D^*$.

We denote by $K(\Sigma)$ the convex set $(\pi^*)^{-1}(K)$.
It is clear that $\Sigma$ is contained in the asymptotic boundary of $K(\Sigma)$.
By Corollary \ref{bound:cor}, $\Sigma$ coincides 
 with the asymptotic boundary of $K(M)$.
 
Clearly no support plane of $K(\Sigma)$ can be timelike. Indeed timelike planes
disconnect the asymptotic boundary of $K(\Sigma)$.
This implies that the boundary of $K(\Sigma)$ in $AdS_{n+1}$ is locally achronal.
Moreover it has two components, and each of them disconnects $AdS_{n+1}$ in
two components.
It follows easily that $K(\Sigma)$ is a convex slub.
\end{proof}

\begin{remark} \label{rk:chull}
The same proof shows that:
for a spacelike graph $M$ in $AdS_{n+1}$, there is convex slub, say $K(M)$, such that
\begin{itemize}
\item $K(M)$ contains $M$.
\item If $K$ is a convex slub containing $M$, then $K(M)\subset K$.
\end{itemize}
The slub $K(M)$ is called the convex hull of $M$.
\end{remark}

Clearly if $D$ is the domain of dependence of $\Sigma$ we have
$K(\Sigma)\subset \overline D$.  An important technical point for what
follows is the following statement. Recall that singular points of
$\Sigma$ are points contained in some light-like segment contained in
$\Sigma$.

\begin{lemma} \label{lm:nosing}
If $\Sigma$ is spacelike graph in $\partial_\infty AdS_{n+1}$ without singular points, 
then the boundary components of $K=K(\Sigma)$ do not contain singular points.
Moreover, in this case, no point of $K$ is contained in $\partial D$.
\end{lemma}

\begin{figure}[h!]
\begin{center}
\input lightlike.pstex_t
\end{center}
\end{figure}

\begin{proof}
Suppose that a lightlike segment $c$ is contained in $\partial_+K$.
Take a support plane $P$ of $\partial_+K$ at some point of $c$. 
Clearly $P$ is lightlike and contains $c$.
For every $p\in c$ notice that
\begin{equation}\label{cc:eq}
I^+(P)\cap \partial_+K=\emptyset~, \qquad \Sigma\subset\overline U_p\,. 
\end{equation}

Let $p_-$ be the past end-point of the lightlike geodesic through $p$
contained in $P$. Let $l$ be the vertical line through $p_-$.  Since
$\Sigma$ is a graph, it must intersect $l$ at some point.  Notice that
one component of $l\setminus\{p\}$ is contained in $I^+(P)$ whereas
the other component is contained in $I^-(p)$. This remark and (\ref{cc:eq}) 
show that $\Sigma$ must
intersect $l$ at $p_-$, that is, $p_-\in\Sigma$.


By a classical theorem on convex sets in Euclidean space (still using
the projective map $\pi^*$ as in Lemma \ref{lm:omega}), $P\cap
K(\Sigma)$ is the convex hull of $P\cap \Sigma$.  Thus there is
another point $q\in P\cap\Sigma$.

By Lemma \ref{llkplanes:lem}, we conclude that $p_-$ and $q$ are
connected by a lightlike segment and this contradicts the assumption
that $\Sigma$ does not contain any singular point.

Eventually, segments joining points of $\partial_+K(\Sigma)$ to $\Sigma$ are spacelike.
By Proposition \ref{boundarydep:prop} we conclude that no point 
of $\partial_+K(\Sigma)$ is contained in  $D$.
\end{proof}

\subsection{Existence of entire maximal graph with given boundary condition}

Let $\Sigma$ be a spacelike graph in $\partial_\infty AdS_{n+1}$ without singular points. 
In this section we prove the main theorem on the existence of a maximal graph with given
asymptotic boundary.

\begin{theorem}\label{maximal:teo}
There is a maximal graph $M$ in $AdS_{n+1}$ whose boundary at infinity coincides with $\Sigma$.
\end{theorem}

Let us consider the following notation that we will use through this section:
\begin{itemize}
\item $D$ is the domain of dependence of $\Sigma$;
\item $K$ is the convex hull of $\Sigma$;
\item $S$ is the future boundary of $K$;
\item $B_r$ is the ball in $\mathbb H^n$ centered at $x^0$ of radius $r$;
\item $S_r$ is the intersection of $S$ with the cylinder $B_r\times \mathbb R$.
\end{itemize}

In \cite{bartnik-regularity} (Theorem 4.1) it is shown that there is
a maximal surface $M_r$ such that $\partial M_r=\partial S_r$.
Moreover $M_r$ is homotopic to $S_r$  (rel. $\partial S_r$)
in the sense that there exists a family of spacelike embeddings
\[
   h_s:S_r\rightarrow AdS_{n+1}
\]
such that 
\begin{enumerate}
\item $h_0=Id$, $h_1(S_r)=M_r$;
\item $h_s(x)=x$ for $x\in\partial S_r$ and $s\in [0,1]$;
\item the map $s\mapsto h_s(x)$ is a vertical path for every $x\in S_r$.
\end{enumerate}

It easily follows that $M_r$ is the graph of some function defined on $B_r$. 
Putting the previous results together we obtain the following lemma.

\begin{lemma}\label{approximation:lem}
For every $r>0$, there is a maximal surface $M_r$ such that $\partial
M_r=\partial S_r$.  Moreover, the surface $M_r$ is a graph of a
function $u_r$ defined on $B_r$ and is contained in $K$.
\end{lemma}

The basic idea of the proof of Theorem \ref{maximal:teo} is to
construct a sequence $r_k\rightarrow+\infty$ such that $u_{r_k}$
converges $C^2$ on compact subset of $\mathbb H^n$.  The proof is
based on an a-priori gradient estimate, that is a particular case of an
estimate proved by Bartnik \cite{bartnik-regularity}.  Given a point
$p\in AdS_{n+1}$ and $\epsilon>0$ we denote by $I_\epsilon^+(p)$ the
set of points in the future of $p$ whose distance from $p$ is at least
$\epsilon$.

\begin{lemma}\label{apriori:lem}
Let $p\in AdS_{n+1}$ and $\epsilon>0$, and let 
$H \subset I^-(p_+)$ be a compact domain 
(where $p_+$ is defined in Section \ref{geodesics:sec}). 
There is a constant $C=C(p,\epsilon, H)$ such that,
for every maximal graph $M$ that verifies the following conditions:
 \begin{itemize}
 \item  $\partial M\cap I^+(p)=\emptyset$,
 \item $M\cap I^+(p)$ is contained in $H$,
  \end{itemize}
we have that
  \[
      \sup_{M\cap I^+_\epsilon(p)} v_M<C
  \]
where $v_M$ is the gradient function of $M$.    
\end{lemma}

\begin{proof}
Let us consider the time-function
\[
  \tau(x)=\delta(x,p)-(\epsilon/2)
\]
where $\delta(x,p)$ is the Lorentzian distance between $x$ and $p$.
This function is smooth on the domain 
$\mathcal V= H\cap I^+(p)$.

Notice that by the assumption on $M$, the region
$M\cap\mathcal V$ contains the region of $M$ where $\tau\geq 0$
and $M\cap I^+_\epsilon(p)$ is contained in $\mathcal V$.

We can apply Theorem 3.1 of \cite{bartnik-regularity} and conclude that
 \[
      \sup_{M\cap I^+_\epsilon(p)} v_M<C
  \]
where $C$ depends on the $C^2$-norms of $t$ and $\tau$ and  on the $C^0$ norm of
$\mathrm{Ric}$, taken on the domain $\mathcal V_{\tau\geq 0}$ 
with respect to a reference Riemannian
metric.
\end{proof}

We can prove now Theorem \ref{maximal:teo}.

\begin{proof}[Proof of Theorem \ref{maximal:teo}]
For every  point $p\in D\cap I^-(\partial_-K)$ we choose $\epsilon=\epsilon(p)$ such that
the family $\{I^+_{\epsilon(p)}(p)\cap K\}_{p\in D\cap I^-(\partial_-K)}$ is an open 
covering of $K$.

Given a number $R$, the intersection $(B_R\times\mathbb R)\cap K$ is compact, 
so there is a finite numbers of points $p_1,\ldots, p_{k_0}\in D\cap I^-(\partial_-K)$ 
such for all $k\in \{ 1,\cdots, k_0\}$, there exists $\epsilon_k=\epsilon(p_k)$ such that
\[
    (B_R\times\mathbb R)\cap K\subset\bigcup_{1}^{k_0} I^+_{\epsilon_k}(p_k)~.
\]

For all $k\in \{ 1,\cdots, k_0\}$, $p_k\in D$, so that the intersection 
$\overline{I^+(p_k)}\cap D$ is compact. Moreover, $D\subset  I^-((p_k)_+)$.
It follows that 
the set $H_k=\overline{I^+(p_k)}\cap K$ is compact and contained in $I^-((p_k)_+)$.

By Lemma \ref{apriori:lem}, there is a constant $C_k$, 
such that
\[
  \sup_{M\cap I^+_{\epsilon_k}(p_k)} v_M<C_k
\]
for every maximal surface $M$ that satisfies the following requirements:
\begin{itemize}
\item  $\partial M\cap I^+(p_k)=\emptyset$; 
\item $M\cap I^+(p_k)$ is contained in $H_k$. 
 \end{itemize}

By the compactness of $I^+(p_k)\cap D$, there is $r_0>0$ such that
\[
   I^+(p_k)\subset B_{r_0}\times\mathbb R
\]
for $k=1,\ldots,k_0$.

Let $\{M_r\}$ be the family of maximal surfaces constructed in Lemma \ref{approximation:lem}.
Then $M_r\subset K$. Moreover there exists $r_0>0$ such that, 
for $r>r_0$, $\partial M_r\cap I^+(p_k)=\emptyset$ for 
$k=1,\ldots, k_0$.

It follows that $\sup_{M_r\cap I^+_{\epsilon_k}(p_k)} v_{M_r}\leq C_k$ for $k=1\ldots,k_0$. 
Since $M_r\cap (B_R\times\mathbb R)\subset\bigcup_k I^+_{\epsilon_k}(p_k)$ we conclude that
\begin{equation}\label{grad:eq}
  \sup_{M_r\cap(B_R\times\mathbb R)} v_{M_r}\leq \max\{C_1,\ldots,C_{k_0}\}
\end{equation}
for every $r>r_0$.

Eventually we deduce that for every $R$ there is a constant $C(R)$ such that
the gradient function of $v_{M_r}$ is bounded by $C(R)$ for $r$ sufficiently big.

Take now any divergent sequence $r_i$.  Let $u_i$ be the function
defined on $B_{r_i}$ such that $M_{r_i}=M_{u_i}$.  By comparing
Equation (\ref{meancurvature:eq}) with estimate (\ref{grad:eq}), we
see that the restriction of $u_i$ on $B_R$ is solution of a
uniformly elliptic quasi-linear operator on $B_R$, with bounded
coefficients.

Since $|u_i|$ and $|\bar\nabla u_i|$ are uniformly bounded on $B_R$,
by elliptic regularity theory (see e.g. \cite{gilbarg-trudinger}) the
norms of $u_i$ in $C^{2,\alpha}(B_{R-1})$ are uniformly bounded.  It
follows that the family $u_i$ is precompact in $C^2(B_{R-1})$.

By a diagonal process we extract a subsequence $u_{i_h}$ converging to
a function $u_\infty$ defined on $\mathbb H^n$ in such a way that the
convergence is $C^2$ on compact sets. Since the $u_{i_h}$ are
uniformly spacelike, so is $u_\infty$. Moreover, since it is the $C^2$
limit of solutions of Equation (\ref{meancurvature:eq}), it is still a solution.

As a consequence, $M=M_u$ is a maximal graph.
Since $M$ is a limit of surfaces contained in $K$, it is contained in $K$.
In particular the asymptotic boundary of $M$ is contained in $\Sigma$, 
and so it coincides with $\Sigma$.
\end{proof}

\subsection{Regularity of maximal hypersurfaces} \label{ssc:regularity}

We will now show that if the distance between $K$ and the past
boundary of $D$ is strictly positive, then any maximal surface
contained in $K$ has bounded second fundamental form.

\begin{theorem}\label{bound:teo}
Suppose that there exists $\epsilon>0$ such that, for every $y\in
\partial_-K$, there exists a point $x\in\partial_-D$ such that
$\delta(x,y)\geq\epsilon$.  Then there exists a constant $C>0$,
depending on $\epsilon$, such that the second fundamental form of any
maximal graph contained in $K$ is bounded by $C$.
\end{theorem}

To prove this theorem we will need the following relation between the
boundaries of $D$ and $K$.  The first part of the lemma will be used
in the proof of Theorem \ref{bound:teo}, while the second part will be
necessary below.

\begin{lemma}\label{distance:lem}
Let $\Sigma\subset \dr_\infty AdS_{n+1}$ be space-like graph, let $K=K(\Sigma)$ be its
convex hull, and let $D=D(\Sigma)$ be its domain of dependence. Then:
\begin{enumerate}
\item For all $q\in K$ and $p\in\partial_-D\cap I^-(q)$ we have
that $\delta(p,q)\leq \pi/2$.
\item For all $q\in\partial_+K$ there exists $p\in\dr_-D\cap I^-(q)$
such that $\delta(p,q)=\pi/2$.
\end{enumerate}
\end{lemma}

The proof of the first point in dimension $2+1$ can be found in \cite{benedetti-bonsante}. 
That argument actually applies in every dimension. For the sake of completeness we
sketch the argument here.

\begin{proof}
Since $p\in\partial_-D$, $\Sigma$ is contained in $\overline U_p$ and
$\Sigma\cap (P_+(p)\cup P_-(p))\neq\emptyset$.

Notice that the plane $P_+(p)$ does not disconnect $\Sigma$, so, it is a support plane for $K$.
In particular $K\subset \overline{I^-(P_+(p))}$.
This implies that the distance of every point of $K\cap I^+(p)$ from $p$ is bounded by $\pi/2$,
and proves the first point. Moreover, since $P_+(p)$ is a support plane of $K$, its 
intersection with $\dr_+K$ is non-empty. But for any point $q\in P_+(p)$ we have 
$\delta(p,q)=\pi/2$, and this proves the second point.
\end{proof}

As a consequence we find a bound on the width of the boundary at infinity of a space-like
graph in $AdS_{n+1}$. This estimate is improved for $n=2$ when the boundary at 
infinity is the graph of a quasi-symmetric homeomorphism, see Theorem \ref{tm:qsym}.

\begin{lemma} \label{lm:width}
Let $M\subset AdS_{n+1}$ be a space-like graph. 
Then $w(\dr_\infty M)\leq\pi/2$.  
\end{lemma}

We can now prove  Theorem \ref{bound:teo}. 

\begin{proof}[Proof of Theorem \ref{bound:teo}]
We consider $q_0=(x^0,0)$ and consider the horizontal plane
 $P_0$ passing though $(x^0,\pi/2-\epsilon/2)$, and define
$H_0=\overline{I^+(q_0)\cap I^-(P_0)}$. 

From Lemma \ref{apriori:lem}, we find a constant $C$ (depending on $\epsilon$)
such that
\[
     \sup_{N\cap I^+_{\epsilon/3}(q_0)} v_N<C
\]
for every maximal surface $N$ such that
\begin{enumerate}
\item $\partial N\cap I^+(q_0)=\emptyset$,
\item $N\cap I^+(q_0)\subset H_0$.
\end{enumerate}

Moreover, by applying the elliptic regularity theory as in
the proof of Theorem \ref{maximal:teo}, we see that 
there is another constant, still denoted by $C$, such that 
\[
  \sup_{N\cap I^+_{\epsilon/2}(q_0)} |A|^2<C
\]
for the same class of maximal surfaces.

Now consider a point $p$ on the maximal surface $M$.
By the assumption there is a point $p_0\in\partial_-D$ such that $\delta(p,p_0)>\epsilon$.
We can fix a point $q$ on the segment $[p_0,p]$ such that $\delta(p,q)>\epsilon/2$.

Since $I^+(q)\cap K$ is compact, there is a point $r\in\partial_+K$
that maximizes the distance from $q$.  Lemma \ref{distance:lem} and
the reverse triangle inequality imply that $\bar
s:=\delta(q,r)<\pi/2-\epsilon/2$.

Moreover the plane passing through $r$ and orthogonal to the segment
$[q,r]$ is a support plane $P$ for $K$ (that is $K\subset
\overline{I^-(P)}$).

Now consider an isometry $\gamma$ of $AdS_{n+1}$ such that  $\gamma(q)=(x^0,0)$
and $\gamma(r)= (x^0,\bar s)$.
We have that $\gamma(P)$ is  the horizontal plane
through $(x^0, \bar s)$.
Since $\bar s<\pi/2-\epsilon/2$, $\gamma(P)\subset I^-(P_0)$.
Thus, $\gamma(K)\subset I^-(P_0)$, and  $\gamma(M)\cap I^+(q_0)\subset H_0$.

In particular $\gamma(M)$ satisfies the conditions (1), (2) above and we conclude that
\[
     \sup_{\gamma(M)\cap I^+_{\epsilon/2}(q_0)} |\tilde A|^2<C\,.
\]
where $\tilde A$ denotes the second fundamental form of $\gamma(M)$.

Since $\gamma(p)\in I^+_{\epsilon/2}(q_0)$ we conclude that
\[
   |A|^2(p)=|\tilde A|^2(\gamma(p))<C\,.
\]
where the constant $C$ is independent of the point $p$.
\end{proof}

\begin{cor} \label{cr:bound}
Suppose that $w(K)<\pi/2$. Then there exists $C>0$ such that any 
maximal space-like graph in $K$ has second fundamental form bounded
by $C$.
\end{cor}

\begin{proof}
Let $\epsilon=\pi/2-w(K)$, so that $\epsilon>0$. Let $y\in \dr_-K$. 
Consider a point $z\in \dr_+K\cap I^+(y)$ for which $\delta(y,z)$
is maximal. Then $\delta(y,z)\leq w(K)$ by definition of $w$. 

Let now $\Delta$ be the past-oriented time-like geodesic ray 
starting from $z$ and containing $y$, and let $x$ be its 
intersection with $\dr_-D$. By the definition of $z$, the
space-like plane orthogonal to $\Delta$ at $z$ is a support
plane of $K$ (otherwise $z$ would not maximize $\delta(y,\cdot)$
on $\dr_+K$). 

This shows that $z$ is also a critical point of
$\delta(x,\cdot)$ on $\dr_+K$ and, since $K$ is convex, it
is a maximum of this function on $\dr_+K$. Therefore 
$\delta(x,z)=\pi/2$ by the second point of Lemma \ref{distance:lem}.
Therefore $\delta(x,y)\geq \epsilon$. So we can apply Theorem 
\ref{bound:teo}, which yields the result.
\end{proof}

%% file: rombus.pstex_t
\begin{picture}(0,0)%
\includegraphics{rombus.pstex}%
\end{picture}%
\setlength{\unitlength}{2210sp}%
\begingroup\makeatletter\ifx\SetFigFont\undefined%
\gdef\SetFigFont#1#2#3#4#5{%
  \reset@font\fontsize{#1}{#2pt}%
  \fontfamily{#3}\fontseries{#4}\fontshape{#5}%
  \selectfont}%
\fi\endgroup%
\begin{picture}(11949,5844)(1174,-5248)
\put(2851,-3136){\makebox(0,0)[lb]{\smash{{\SetFigFont{7}{8.4}{\rmdefault}{\mddefault}{\updefault}{$U$}%
}}}}
\put(10216,-916){\makebox(0,0)[lb]{\smash{{\SetFigFont{7}{8.4}{\rmdefault}{\mddefault}{\updefault}{$z_+$}%
}}}}
\put(9946,-4951){\makebox(0,0)[lb]{\smash{{\SetFigFont{7}{8.4}{\rmdefault}{\mddefault}{\updefault}{$z_-$}%
}}}}
\put(2176,-1231){\makebox(0,0)[lb]{\smash{{\SetFigFont{7}{8.4}{\rmdefault}{\mddefault}{\updefault}{$z_+$}%
}}}}
\put(1831,-4936){\makebox(0,0)[lb]{\smash{{\SetFigFont{7}{8.4}{\rmdefault}{\mddefault}{\updefault}{$z_-$}%
}}}}
\put(2191,-3076){\makebox(0,0)[lb]{\smash{{\SetFigFont{7}{8.4}{\rmdefault}{\mddefault}{\updefault}{$x$}%
}}}}
\put(4771,-2551){\makebox(0,0)[lb]{\smash{{\SetFigFont{7}{8.4}{\rmdefault}{\mddefault}{\updefault}{$P_0$}%
}}}}
\put(2776,-4126){\makebox(0,0)[lb]{\smash{{\SetFigFont{7}{8.4}{\rmdefault}{\mddefault}{\updefault}{$L_-$}%
}}}}
\put(9871,-2281){\makebox(0,0)[lb]{\smash{{\SetFigFont{7}{8.4}{\rmdefault}{\mddefault}{\updefault}{$c_n$}%
}}}}
\put(10486,-3001){\makebox(0,0)[lb]{\smash{{\SetFigFont{7}{8.4}{\rmdefault}{\mddefault}{\updefault}{$\Phi(c_n)$}%
}}}}
\put(3631,-3076){\makebox(0,0)[lb]{\smash{{\SetFigFont{7}{8.4}{\rmdefault}{\mddefault}{\updefault}{$l$}%
}}}}
\put(3016,-2026){\makebox(0,0)[lb]{\smash{{\SetFigFont{7}{8.4}{\rmdefault}{\mddefault}{\updefault}{$L_+$}%
}}}}
\put(1966,-331){\makebox(0,0)[lb]{\smash{{\SetFigFont{7}{8.4}{\rmdefault}{\mddefault}{\updefault}{$P_\infty$}%
}}}}
\end{picture}%

%% file: graphs.pstex_t
\begin{picture}(0,0)%
\includegraphics{graphs.pstex}%
\end{picture}%
\setlength{\unitlength}{2901sp}%
\begingroup\makeatletter\ifx\SetFigFont\undefined%
\gdef\SetFigFont#1#2#3#4#5{%
  \reset@font\fontsize{#1}{#2pt}%
  \fontfamily{#3}\fontseries{#4}\fontshape{#5}%
  \selectfont}%
\fi\endgroup%
\begin{picture}(8902,3043)(841,-3712)
\end{picture}%

%% file: lightlike.pstex_t
\begin{picture}(0,0)%
\includegraphics{lightlike.pstex}%
\end{picture}%
\setlength{\unitlength}{1381sp}%
\begingroup\makeatletter\ifx\SetFigFont\undefined%
\gdef\SetFigFont#1#2#3#4#5{%
  \reset@font\fontsize{#1}{#2pt}%
  \fontfamily{#3}\fontseries{#4}\fontshape{#5}%
  \selectfont}%
\fi\endgroup%
\begin{picture}(6516,9124)(4298,-10021)
\put(6763,-5292){\makebox(0,0)[lb]{\smash{{\SetFigFont{5}{6.0}{\rmdefault}{\mddefault}{\updefault}{\color[rgb]{0,0,0}$p$}%
}}}}
\put(8646,-7059){\makebox(0,0)[lb]{\smash{{\SetFigFont{5}{6.0}{\rmdefault}{\mddefault}{\updefault}{\color[rgb]{0,0,0}$I^-(p)$}%
}}}}
\put(9163,-2842){\makebox(0,0)[lb]{\smash{{\SetFigFont{5}{6.0}{\rmdefault}{\mddefault}{\updefault}{\color[rgb]{0,0,0}$I^+(P)$}%
}}}}
\put(4313,-6992){\makebox(0,0)[lb]{\smash{{\SetFigFont{5}{6.0}{\rmdefault}{\mddefault}{\updefault}{\color[rgb]{0,0,0}$p_-$}%
}}}}
\put(4426,-4261){\makebox(0,0)[lb]{\smash{{\SetFigFont{5}{6.0}{\rmdefault}{\mddefault}{\updefault}{\color[rgb]{0,0,0}$l$}%
}}}}
\end{picture}%

%% file: jms2.tex
\section{Uniqueness of maximal surfaces in $AdS_3$}
\label{sc:uniqueness}

We consider in this section the uniqueness of maximal graphs with given
boundary at infinity and bounded second fundamental form in $AdS_3$. 
The argument has two parts. The first is to show that those surfaces
have negative sectional curvature. The second part is to show that the
existence of such a negatively curved maximal space-like graph forbids
the existence of any other maximal graph with the same boundary. 
Both parts use
a version ``at infinity'' of the maximum principle, for which a compactness
argument is needed. For the first part we need a simple compactness 
statement on sequences of maximal surfaces.

\subsection{A compactness result for sequences of maximal hypersurfaces}

The following statement is useful to use ``at infinity'' the maximum principle.

\begin{lemma} \label{lm:compact}
Choose $C>0$, a point $x_0\in AdS_{n+1}$, and a future-oriented 
unit time-like vector $n_0\in T_{x_0}AdS_{n+1}$. There exists $r_0>0$ as follows.
Let $P_0$ be the space-like hyperplane orthogonal to $n_0$ at $x_0$, let $D_0$
be the disk of radius $r_0$ centered at $x_0$ in $P_0$, and let
$(S_n)_{n\in \N}$ be a sequence of maximal space-like graphs containing $x_0$ and
orthogonal to $n_0$, with second fundamental form bounded by $C$. 
After extracting a sub-sequence, the restrictions of
the $S_n$ to the cylinder above $D_0$ converge $C^\infty$ to a maximal
space-like disk with boundary contained in the cylinder over $\dr D_0$.
\end{lemma}

The proof given here applies with a few modifications to the more general 
context of maximal (resp. minimal) immersions of hypersurfaces in any
Lorentzian (resp. Riemannian) manifold with bounded
geometry, we state the lemma in $AdS_{n+1}$ for simplicity. 

\begin{proof}
For all $n$, the surface $S_n$ is the graph of a function $f_n$ over $P_n$. 
The bound on the second fundamental form of $S_n$, along with the
fact that the $S_n$ are orthogonal to $n_0$, indicates that, for some $r>0$,
the derivative of $f_n$ is bounded on the disk of center $x_0$ and radius $r$,
more precisely there exists $\epsilon>0$ such that 
$$ \phi\|\nabla f_n\|<1-\epsilon $$
on this disk of center $x_0$ and radius $r$.

This, along with the bound on the second fundamental form of $S_n$ (again)
shows that the Hessian of $f_n$ is bounded by a constant depending on $r$
(for $r$ small enough). Thus we can extract from $(f_n)_{n\in \N}$ a 
subsequence which is $C^{1,1}$ converging to a function $f_\infty$ on the
disk of center $x_0$ and radius $r$. Moreover the gradient of $f_\infty$ is 
uniformly bounded, so that the graph of $f_\infty$ is a disk which is uniformly
space-like.

By definition the $f_n$ are solutions of Equation (\ref{meancurvature:eq}), which just
translates analytically the fact that their graphs are maximal surfaces.
Since $f_\infty$ is a $C^{1,1}$-limit of the $f_n$, it is itself a weak
solution of (\ref{meancurvature:eq}). Since Equation (\ref{meancurvature:eq}) 
is quasi-linear, it then follows from 
elliptic regularity that $f_\infty$ is $C^\infty$, and that $(f_n)$
is $C^\infty$-converging to $f_\infty$ (see \cite{gilbarg-trudinger}). 
This means that the restriction 
of the $S_n$ to the cylinder above the disk of radius $r_0$ in $P_0$, 
for some $r_0>0$ (depending only on $C$) converge to a limit which
is a maximal surface, the graph of $f_\infty$ over the disk of
radius $r_0$.
\end{proof}

\subsection{Maximal surfaces with bounded second fundamental form}
\label{ssc:52}

The first proposition of this section is the following, its proof is based
on Lemma \ref{lm:compact}.

\begin{prop} \label{pr:maxk}
Let $S$ be a complete maximal surface in $AdS_3$. Suppose that the norm of the 
fundamental form of $S$ is bounded. Then $S$ either has negative sectional curvature,
or $S$ is flat. If the supremum of the sectional curvature of $S$ is $0$, then 
$w(\dr_\infty S)=\pi/2$.
\end{prop}

The completeness mentioned here is with respect to the induced metric
on $S$.  The proof uses two preliminary statements. The first is taken
from \cite{minsurf}, where it can be found in the proof of Lemma 3.11,
p. 214. Note that the sign of the Laplacian used here is defined so
that $\Delta$ is negative as an operator acting on $L^2$.

\begin{lemma} \label{lm:equation}
Let $\Sigma$ be a maximal space-like surface in a 3-dimensional AdS
manifold. Let $B$ be its shape operator, and let
$\chi=\log(-\det(B))/4$. Then $\chi$ satisfies the equation
$$ \Delta \chi=e^{4\chi}-1~. $$
\end{lemma}

As a consequence, we can apply the maximum principle to $\chi$, it shows
that $\chi$ cannot have a positive local maximum. This can be translated into
a statement on $K$, using the Gauss formula, which shows that $K=-1+e^{4\chi}$.

\begin{lemma} \label{lm:max_prin}
Suppose that $K$ has a local maximum at a point where it is non-negative.
Then $K=0$ at that point, and on the whole surface $S$, so that $S$ is flat
(in the intrinsic sense).
\end{lemma}

We need another elementary statement, characterizing the maximal
surfaces with flat induced metric in $AdS_3$. We include the proof for
the reader's convenience.

\begin{lemma} \label{lm:horo}
Let $\Sigma$ be a space-like maximal surface in $AdS_3$, with zero
sectional curvature. Then $\Sigma$ is a subset of a ``horosphere'',
that is, its principal curvatures are $-1$ and $1$, and its lines of
curvature form two orthogonal foliations by parallel lines. If
$\Sigma$ is a space-like graph, then its boundary at infinity is the
union of four light-like segments in $\dr_\infty AdS_3$.
\end{lemma}

\begin{proof}
Since $\Sigma$ is maximal, its principal curvatures are at each point two
opposite numbers, $k$ and $-k$. The Gauss formula asserts that the sectional
curvature of $\Sigma$ is $K=-1+k^2$, so $k=1$. Let $(e_1,e_2)$ be an
orthonormal frame of unit principal vectors on $\Sigma_0$, and let $\II$
be the second fundamental form of $\Sigma$. The Codazzi equation can be written
as follows, at any point $m\in \Sigma$, for any vector field $x$ on $\Sigma$ such
that $\nabla x=0$ at $m$:
$$ I((d^\nabla B)(e_1,e_2),x)= e_1.\II(e_2,x)-e_2.\II(e_1,x)-\II([e_1,e_2],x)=0~. $$
Since the first two terms clearly vanish and $\II$ is non-degenerate, 
$[e_1,e_2]=0$, so that, if $\omega$ is the connection form of the frame
$(e_1, e_2)$, 
$$ \nabla_{e_1}e_2-\nabla_{e_2}e_1 = -\omega(e_1)e_1 -\omega(e_2)e_2=0~. $$
Therefore $e_1$ and $e_2$ are both parallel vector fields, and the 
first part of the statement follows.

There is a simple way to describe such a horosphere. Consider a
space-like line $\Delta$ in $AdS_3$, and the set $\Sigma_0$ of
endpoints of the future-oriented time-like segments of length $\pi/4$
starting from $\Delta$. An explicit computation (as in the proof of
Proposition \ref{pr:maxk} below) shows that $\Sigma_0$ is precisely a
horosphere as described above. The action of the isometry group of
$AdS_3$ shows that there exists a unique surface of this type passing
through each point $x$ of $AdS_3$, with fixed (time-like) normal and fixed
principal direction at $x$ for the principal curvature $+1$, so any
maximal graph with zero sectional curvature is of this type.

Let $\Delta^*$ be the line dual to $\Delta$, that is, the set of endpoints of 
future-oriented time-like segments of length $\pi/2$ starting from $\Delta$
(see Section \ref{proj:sec}).
Now let $\dr\Sigma_0$ be the boundary at infinity of $\Sigma_0$. Considering
the projective model of $AdS_3$ shows that $\dr\Sigma_0$ 
contains the endpoints at infinity $\Delta_-$ and $\Delta_+$ 
of $\Delta$, and also the endpoints at infinity of $\Delta^*_+$ and
$\Delta^*_-$ of $\Delta^*$. Since $\dr\Sigma_0$ is a nowhere time-like curve
in $\dr_\infty AdS_3$, it is necessarily made of the four segments from 
$\Delta_+$ to $\Delta^*_+$, from $\Delta_+^*$ to $\Delta_-$, from $\Delta_-$
to $\Delta^*_-$, and from $\Delta^*_-$ to $\Delta_+$, which are all 
light-like. This proves the last part of the lemma.
\end{proof}

\begin{proof}[Proof of Proposition \ref{pr:maxk}]
Since $S$ has bounded second fundamental form, its sectional curvature
$K$ is bounded, we call $K_S$ the upper bound of $K$ on $S$. Lemma
\ref{lm:max_prin} already shows that if this upper bound is attained
on $S$, then it is non-positive, and if it is equal to $0$ then $S$ is
flat. We will use Lemma \ref{lm:compact} to extend this argument to
the case where the upper bound $K_S$ is not attained.

Consider a sequence $(s_n)_{n\in \N}$
of points in $S$ such that $K_S-1/n<K(s_n)<K_S$, and apply to $S$ a sequence of
isometries $(\phi_n)_{n\in \N}$ which sends $s_n$ to a fixed point $x_0$ and the
oriented unit normal vector to $S$ at $s_n$ to a fixed vector $n_0$. 
Since $S$ has bounded second fundamental form, Lemma \ref{lm:compact} shows that
we can extract from the sequence $(\phi_n(S))_{n\in \N}$ a subsequence which
converges, in the neighborhood of $x_0$, to a maximal space-like graph $S_0$.
By construction the curvature of $S_0$ has a local maximum at $x_0$, and
this local maximum is equal to $K_S$. Lemma \ref{lm:max_prin} therefore shows
that $K_S\leq 0$. 

Suppose now that $K_S=0$. Then the sequence $\phi_n(S)$ converges, in
a neighborhood of $x_0$, to a ``horosphere'' $\Sigma_0$, as described
in Lemma \ref{lm:horo}.  Lemma \ref{lm:compact} shows that the
convergence is $C^\infty$ in compact subsets of $AdS_3$.  Let
$E_n$ be the boundary at infinity of $\phi_n(S)$. Since
$\phi_n(S)$ is space-like, $E_n$ is a nowhere time-like curve in
$\dr_\infty AdS_3$.  By construction,
$E_n=(\rho_{l,n},\rho_{r,n})E$, where $E=\dr_\infty S$,
$(\rho_{l,n})$ and $(\rho_{r,n})$ are two sequences of elements of
$PSL_2(\R)$, and, for all $n\in \N$, $(\rho_{l,n},\rho_{r,n})$ is
considered as an isometry acting on $AdS_3$ through the natural
identification (see Section \ref{ssc:3d} or \cite{mess,mess-notes}).

By Lemma \ref{lm:boundary} (more precisely the fact that space-like
hypersurfaces in $AdS_{n+1}$ are the graphs of 2-Lipschitz functions), 
since $\phi_n(S)$ converges on compact
subsets of $AdS_3$ to $\Sigma_0$, $E_n$ converges to the boundary
at infinity of $\Sigma_0$, which we call $E_0$. In particular,
using the notations in the proof of Lemma \ref{lm:horo}, for each
$n\in \N$ there are four points $x^+_n, x^-_n, x^{+*}_n, x^{-*}_n\in
E_n$ which can be chosen so that $x^+_n\rightarrow \Delta_+$,
$x^-_n\rightarrow \Delta_-$, $x^{+*}_n\rightarrow \Delta^*_+$ and
$x^{-*}_n\rightarrow \Delta^*_-$.

Therefore, for $n$ large enough, there are points $y_n, z_n$ which are 
arbitrarily close to $\Delta$ and to $\Delta^*$ respectively, with 
$(y_n)$ and $(z_n)$ converging to limits respectively in $\Delta$ and
to $\Delta^*$. The distance between the limits is $\pi/2$, so that the
distance between $y_n$ and $z_n$ goes to $\pi/2$ as $n\rightarrow \infty$,
this shows that $w(K)=\pi/2$.
\end{proof}

\subsection{Quasi-symmetric homeomorphisms and the width} 

There is another important relation which is valid only in $AdS_3$, as stated
in the next proposition.

\begin{prop} \label{pr:width}
Let $E$ be a weakly space-like graph in $\dr_\infty AdS_3$ (that is, $E$
is a space-like curve). Let $K$ be the convex hull of $E$. Suppose that
$w(K)=\pi/2$. Then $E$ is not the graph of a quasi-symmetric 
homeomorphism from $S^1$ to $S^1$.
\end{prop}

\begin{proof}
We suppose that $w(K)=\pi/2$, it follows that there exist two sequences
of points $(x_n)$ in $\dr_-K$ and $(y_n)$ in $\dr_+K$ such that 
$\delta(x_n,y_n)\rightarrow \pi/2$. We can suppose (replacing $x_n$
and $y_n$ by points in the same face of $\dr K$ if necessary) that
$x_n$ is contained in a space-like geodesic $\Delta_n\subset\dr_-K$,
and that $y_n$ is contained in a space-like geodesic 
$\Delta'_n\subset \dr_+K$.

We can find a sequence $(\phi_n)$
of isometries of $AdS_3$ such that $\phi_n(x_n)\rightarrow x$, 
$\phi_n(y_n)\rightarrow y$, with $\delta(x,y)=\pi/2$. Moreover, 
$\phi_n(K)$ is the convex hull $\phi_n(E)$. Since the 
$\phi_n(K)$ are convex, they converge (perhaps after extracting a 
subsequence) in the Hausdorff topology to a limit $K_0$, which
is the convex hull of $E_0=\lim \phi_n(E)$. Moreover,
extracting a subsequence again if necessary, we can suppose that 
$\phi_n(\Delta_n)\rightarrow \Delta$ and that $\phi_n(\Delta'_n)
\rightarrow \Delta'$. Since $x\in \Delta$, $y\in \Delta'$, and
$\delta(x,y)=\pi/2$, $\Delta'=\Delta^*$, otherwise the width of
$K_0$ would have to be strictly larger than $\pi/2$, contradicting
Lemma \ref{lm:width}.

Then $E_0$ contains the endpoints $\Delta_-,\Delta_+$ of 
$\Delta$, and the endpoints $\Delta^*_-,\Delta^*_+$ of $\Delta^*$.
Since $E$ is weakly space-like, so is $E_0$, so it is
the union of four light-like segments joining those four points.

Since $E_0$ is composed of four light-like segments (with endpoints 
$\Delta_+,\Delta^*_+,\Delta_-$ and $\Delta^*_-$) there are points 
$u,v$ and $u',v'$ in $\R P^1$, with $u\neq v$ and $u'\neq v'$, such that,
in the identification of $\dr_\infty AdS_3$ with $\R P^1\times \R P^1$, 
$\Delta_+=(u,u')$, $\Delta^*_+=(u,v')$, $\Delta_-=(v,v')$, and 
$\Delta^*_-=(v,u')$. 

So $E_0$ is the graph of the function $f_0:\R P^1\rightarrow
\R P^1$ sending $(u,v)$ to $v'$ and $(v,u)$ to $u'$. After composing
on the right and on the left with projective transformations, we can
suppose that it is the graph of the function $f_0:\R P^1\rightarrow \R P^1$
sending $(0,2)$ to $0$ and $(2,\infty]\cup[-\infty,0)$ to $1$. 

Consider the points $-3,-1,1,\infty\in \R P^1$. A direct computation shows
that their cross-ratio is $[-3,-1;1,\infty]=2$, while the cross-ratio of
their images by $f_0$ is $[0,0;1,1]=1$. 

It follows that there are 4-tuples
of points on $\phi_n(E)$ whose projection by $p_l$ are 4-tuples of points
with cross-ratio arbitrarily to $2$ and whose projection by $p_r$
are 4-tuples of points with cross-ratio arbitrarily close to $1$. This
means precisely, by definition of a quasi-symmetric homeomorphism, that $E$ is not
the graph of a quasi-symmetric homeomorphism. 
\end{proof}

\subsection{Uniqueness of negatively curved maximal surfaces}

We now turn to the second proposition of this section, the fact that maximal
space-like graphs with negative sectional curvature are uniquely determined,
among all maximal space-like graphs, by their boundary at infinity. 

\begin{prop} \label{pr:k1}
Let $S$ be a maximal graph in $AdS_3$, with sectional curvature bounded from
above by a negative constant.
Then $S$ is unique among complete 
maximal graphs with given boundary curve at infinity
and bounded second fundamental form.
\end{prop}

We first state a preliminary lemma (see also Lemma \ref{lm:ch}). Note that
from this point on we will often consider space-graphs in the projective model
of $AdS_{n+1}$.

\begin{lemma} \label{lm:saddle}
Let $u:S^1\rightarrow S^1$ be a homeomorphism, and let $E_u\subset S^1\times S^1
\simeq \dr\pi(AdS_3)$ be its graph. Let $C(E_u)$ be defined as in the paragraph before
Definition \ref{df:width}. Then any maximal surface in $AdS_3$ with boundary at
infinity $E_u$ is contained in $C(E_u)$.
\end{lemma}

\begin{proof}
Let $S\subset AdS_3$ be a maximal surface, with boundary at infinity $E_u$. 
The image of $S$ in the projective model of $AdS_3$ is a saddle surface, that is,
a surface which has opposite principal curvatures at each point. A characterization
of saddle surfaces (see \cite[Section 6.5.1]{burago-zalgaller:geometric}) is that, for any
relatively compact subset $G\subset S$, then $G$ is contained in the convex hull
of $\dr G$. This property, applied to an exhaustion of the image of $S$ in the projective
model by compact subsets, is precisely what we need.
\end{proof}

\begin{proof}[Proof of Proposition \ref{pr:k1}]
We consider the domain $\Omega$ introduced in the proof of Claim \ref{cl:equidistant},
as the set of points at time-like distance at most $\pi/4$ from $S$. Claim
\ref{cl:equidistant} shows that $\Omega$ is convex, with smooth,space-like boundary.

Consider now another maximal graph $S'\subset AdS_3$, complete, with the same 
boundary at infinity as $S$, and with bounded second fundamental form. 
By construction the boundary of $\Omega$ is equal to $E$. 
Since $\Omega$ is convex, it contains the convex hull of $E$ and
therefore, by Lemma \ref{lm:saddle}, it contains $S'$. Let $r_1$ be
the supremum over $S'$ of the distance to $S$. The argument above
shows that $r_1\in [0,\pi/4-\alpha)$, and the maximum principle shows
that, if $r_1>0$, then it cannot be attained at an interior point of $S'$, since 
then $S'$ would have to be tangent from the interior of $S_{r_1}$, 
which would contradict the maximality of $S'$.

Since $S'$ is complete, there exists a sequence
$(x_n)_{n\in \N}$ of points in $S'$ such that
$d(x_n,S)\rightarrow r_1$ and that the norm of the 
differential at $x_n$ of the restriction to $S'$ of the 
distance to $S$ goes to zero as $n\rightarrow \infty$
(this is a very weak form of a lemma appearing e.g. 
in \cite{yau:cpam75}).

Consider a sequence
of isometries $(\phi_n)_{n\in \N}$ chosen such that $\phi_n(x_n)$ is 
equal to a fixed point $x_0$, and that the normal to $\phi_n(S')$ at
$\phi_n(x_n)$ is a fixed vector $n_0$. 
Lemma \ref{lm:compact} shows that, after extracting a sub-sequence,
$(\phi_n(S'))_{n\in \N}$ converges in a neighborhood of $x_0$ to 
a smooth, maximal surface $S'_\infty$. Moreover, since the differential
at $x_n$ of the distance to $S$ goes to zero, the images by $\phi_n$
of $S$ also converge to a limit $S_\infty$, in a neighborhood of its
intersection with the normal to $S'_\infty$ at $x_0$. 

We can now apply the maximum principle to the distance to $S'_\infty$
as a maximal surface in the foliation by the surfaces equidistant to 
$S_\infty$, and obtain a contradiction if $r_1>0$. So $r_1=0$, 
and $S'=S$.
\end{proof}

Together with Proposition \ref{pr:maxk} and Proposition \ref{pr:width}, Proposition \ref{pr:k1} 
leads directly to a simple consequence.

\begin{cor}
Let $S$ be a maximal graph in $AdS_3$, with bounded second fundamental form. 
Suppose that the boundary at infinity of $S$ is the graph $E$ of a quasi-symmetric
homeomorphism from $S^1$ to $S^1$. 
Then $S$ is the unique maximal surface with boundary at infinity $E$ and bounded
second fundamental form.
\end{cor}

\section{Proof of the main results}
\label{sc:proof}

\subsection{A characterization of quasi-symmetric homeomorphisms}
\label{ssc:qsym}

We now prove Theorem \ref{tm:qsym}.
Let $u:S^1\rightarrow S^1$ be a homeomorphism, and let $E_u$ be its graph.
We already know, from Lemma \ref{lm:width}, that $w(E_u)\leq \pi/2$. 
Moreover Proposition \ref{pr:width} shows that if $u$ is quasi-symmetric, then
$w(E_u)<\pi/2$. 

Suppose conversely that $w(E_u)<\pi/2$. 
We can apply Theorem \ref{maximal:teo} to $E_u$, and obtain a
maximal graph $M$ in $AdS_3$ with boundary at infinity equal
to $E_u$. Corollary \ref{cr:bound} shows that $M$ has bounded
second fundamental form.

Proposition \ref{pr:maxk} 
then shows that $M$ has sectional curvature bounded from above
by a negative constant. Therefore we obtain through Proposition \ref{pr:translation}
a minimal Lagrangian quasi-conformal diffeomorphism $\phi$ with 
boundary value equal to $u$. Since $\phi$ is quasi-conformal, 
$u$ is quasi-symmetric, as claimed.

\subsection{Theorems \ref{tm:ml} and \ref{tm:uniqueness}}

Theorem \ref{tm:ml} clearly follows, through Proposition \ref{pr:translation}, 
from Theorem \ref{tm:uniqueness},
so we now concentrate on this last statement.

\begin{proof}[Proof of Theorem \ref{tm:uniqueness}]
Let $E=\dr_\infty S\subset\dr_\infty AdS_3$, and let $M$ be the maximal
graph with boundary at infinity $E$ which is provided by Theorem 
\ref{maximal:teo}. Since $E$ is the graph of a quasi-symmetric homeomorphism,
Proposition \ref{pr:width} shows that $w(E)<\pi/2$. 

The argument in the previous paragraph then shows that $E$ is the boundary
at infinity of a maximal graph $M$ in $AdS_3$, which has bounded 
second fundamental form by Theorem \ref{bound:teo}. Then Proposition 
\ref{pr:maxk} shows that $M$ has sectional curvature bounded from
above by a negative constant. Proposition \ref{pr:k1} can therefore be
used to obtain that $M$ is unique among maximal graphs with 
boundary at infinity $E$ and bounded second fundamental form. 
\end{proof}

%% file: fb2.tex
\section{Mean curvature flow for spacelike graphs}

In this section we prove a longtime existence solution 
for the mean curvature flow of spacelike graphs in $AdS_{n+1}$.
The proof is based on Ecker's estimates \cite{ecker}, that are the parabolic
analogous of Bartnik's estimates we have used in Lemma \ref{apriori:lem}.
This argument provides an alternate proof of the existence and 
regularity of maximal surfaces with given asymptotic boundary already
proved in Section \ref{sc:maximal}.

We recall that a mean curvature flow of a spacelike surface is 
a family of spacelike embeddings $\sigma_s:M\rightarrow AdS_{n+1}$
such that
\begin{equation}\label{mcf:eq1}
\frac{\partial \sigma}{\partial s}(x,s)=H(x,s)\nu(x,s)
\end{equation}
where $H(x,s)$ and $\nu(x,s)$ are respectively
the mean curvature and the normal vector of the surface 
$M_s=\sigma_s(S)$ at point $\sigma_s(x)$.

We also consider the case where $M$ is compact with boundary. In that case 
we always consider the Dirichlet condition
\begin{equation}\label{dir:eq}
\sigma_s(x)=\sigma_0(x)\qquad\textrm{ for all }x\in\partial M~.
\end{equation}

\begin{lemma}\label{graph:lem}
Let $(M_s)_{s\in[0,s_0]}$ be a family of spacelike surfaces moving
by mean curvature flow.
If $M_0$ is a graph of a function $u_0$  defined on some domain $\Omega$ of $\mathbb H^n$
with smooth boundary,
then so is $M_s$ for every $s\in[0,s_0]$.

Moreover, if $u_s:\Omega\rightarrow\mathbb R$ is the function defining $M_s$ then 
\begin{equation}\label{mcf:eq2}
\frac{\partial u}{\partial s}=\phi^{-1}v^{-1}H
\end{equation}
where $v$ is the gradient function on $M_s$
\end{lemma}
\begin{proof}
Since $M_s$ is homotopic to $M_0$ through a family of
spacelike surfaces with fixed boundary , then $M_s$ is contained
in the domain of dependence of $M_0$ 
that, in turn,  is contained in $\Omega\times\mathbb R$.

Moreover, $M_s$ disconnects $\Omega\times\mathbb R$ in two regions.
The same argument as in Proposition \ref{spacelikegraphs:prop} shows that
$M_s$ is a graph on $\Omega$ of a function $u_s$.

The evolution equation of $u_s$ is computed in \cite{ecker-huisken}.
\end{proof}

\begin{remark}
\begin{enumerate}
\item Notice that  $\frac{\partial (t\circ\sigma)}{\partial s}=\phi^{-1}v H$, that is different from
(\ref{mcf:eq2}).
The reason is that the curve $\sigma(x,\cdot)$ at some point $s$ is tangential to the normal 
of $M_s$, so in general it is not a vertical line. This implies that the function
$u_s$ agrees with $t|_{M_s}$ only up some tangential diffeomorphism of $M_s$.
\item
Equation (\ref{mcf:eq2})  is equivalent, up to tangential diffeomorphisms, to equation
(\ref{mcf:eq1}). This means that if $(u_s)_{s\in [0,s_0]}$ is a solution of (\ref{mcf:eq2}), there is a time-dependent field $X_s$ on $\Omega$ such that the map
$\sigma:\Omega\times[0,s_0]\rightarrow AdS^{n+1}$ defined by
\[
\sigma(x,s)=(\psi_s(x), u_s(\psi_s(p)))
\]
is a solution of (\ref{mcf:eq1}), where $\psi_s$ is the flow of $X_s$.
\end{enumerate}
\end{remark}

\begin{prop}\label{compactmcf:prop}\cite{ecker}
Let $M_0$ be a spacelike $\mathrm C^{0,1}$  compact graph in $AdS^{n+1}$.
Then there is a smooth solution of (\ref{mcf:eq1}) for $s\in(0,+\infty)$ such that
\begin{itemize}
\item  $\partial M_s=\partial M_0$ for every $s$;
\item $M_s\rightarrow M_0$ in the Hausdorff topology as $s\rightarrow 0$;
\item $M_s\rightarrow M_\infty$ in the $\mathrm C^{\infty}$-topology as $s\rightarrow +\infty$,
where $M_\infty$ is the unique maximal spacelike surface with the property that
$\partial M_\infty=\partial M_0$;
\item   if $H_s$ denotes the mean curvature on $M_s$ we have
\begin{equation}\label{mcest:eq}
H^2_s(x)\leq \frac{n}{2}\frac{1}{s}\,.
\end{equation}
\end{itemize} 
\end{prop}

\subsection{Mean curvature flow and convex subsets}

To show the convergence of the mean curvature flow, we need to remark that, under
suitable hypothesis, it does not leave convex subsets of $AdS_{n+1}$. 

\begin{lemma}\label{mcfplanes:lem}
Let $M_s$ be a compact solution of  (\ref{mcf:eq1}).
Suppose that there exists a spacelike plane  $P$ such that
$M_0$ is contained in $\overline{I^-(P)}$ and $\partial M_0\subset I^-(P)$. 
Then $M_s$ is contained in $I^-(P)$ for every $s>0$.
\end{lemma}

\begin{proof}
Without loss of generality we can suppose that $P$ is the horizontal plane.
We consider the function $u:AdS^{n+1}\rightarrow\R$ 
defined, as in the proof of Lemma \ref{planes:lem},
by $u(x,t)=x_{n+1}\sin t$.
 
By our assumption
\begin{equation}\label{neg:eq}
\begin{array}{ll}
u(p)\leq 0 & \textrm{for every } p\in M_0~,\\
u(p)< 0 &\textrm{for  every }p\in\partial M_s~.
\end{array}
\end{equation}

On the other hand the computation in Lemma \ref{planes:lem}
shows that
\[
  (\frac{d}{ds}-\Delta)u=-nu
\]
where $\Delta$ is the Laplace-Beltrami operator
on $M_s$.

In particular if the maximum of the function $u$ is achieved at some interior point of
$M_s$ we have
\[
\frac{du_{max}}{ds}\leq nu_{max}~.
\]
By (\ref{neg:eq}), we deduce that $u_{max}(s)<0$ for every $s>0$.
In particular $M_s$ is contained in the region $\{(x,t)|0<t<\pi\}$ for every $s>0$.
\end{proof}

Lemma \ref{mcfplanes:lem} and Lemma \ref{lm:omega}
imply the following property.

\begin{prop}\label{mcfconvexity:prop}
If $M_s$ be a compact solution of (\ref{mcf:eq1}) such that
$M_0$ is contained in the closure of some convex slub $\Omega$, 
and $\partial M_0$ is contained in $\Omega$, 
then $M_s$ is contained in $\Omega$ for every $s>0$.
\end{prop}


Let $M=\Gamma_u$ be a weakly 
spacelike  graph and $\Sigma$ be 
its asymptotic boundary.
We will assume that neither $M$ nor $\Sigma$ contains any
singular point.
Finally we denote by $D$
the domain of dependence of $M$ and by $K$ its convex hull,
introduced in Remark \ref{rk:chull}.
The same argument as in Lemma \ref{lm:nosing}
shows that $\overline{K}\cap \dr D=\emptyset$. 

For every $r>0$ let $u^r$ be the restriction of $u$ on 
 $B_r$ (that is the ball in $\mathbb H^n$ of center at $x^0$ and 
radius $r$).

We consider the mean curvature flow with Dirichlet condition of the compact graph
of $u^r$, that is, a map
\[
\sigma^r: \overline B_r\times (0,+\infty)\rightarrow AdS^{n+1}
\]
that verifies (\ref{mcf:eq1}) and satisfies
\begin{itemize}
\item $\sigma^r(x,0)=(x,u(x))$ for every $x\in B_r$;
\item  $\sigma^r(x,s)=(x,u(x))$ for every $x\in \partial B_r$.
\end{itemize}
Let us denote by $M^r_s$ the image of $B_r$ through the map $\sigma(\cdot,s)$.

By Lemma \ref{graph:lem} and Proposition \ref{compactmcf:prop}
 there is a family of spacelike functions
\[
   u^r_s:\overline B_r\rightarrow \mathbb R
\]
such that $M^r_s$ is the graph of $u^r_s$ and the family $(u^r_s)$ 
satifies (\ref{mcf:eq2}).



\begin{prop}\label{mainmcf:prop}
For every $R>0$, $\eta>0$ 
there is $\bar r>0$ and  constants $C, C_0,C_1,\ldots$ such that for every $r>\bar r$  and every 
$s>\eta$ we have
\[
\begin{array}{l}
    \sup_{M^r_s\cap B_R\times\mathbb R} v<C\\
    \sup_{M^r_s\cap B_R\times\mathbb R}|\nabla^m A|^2<C_m\qquad\textrm{for }m=0,\ldots\ .
\end{array}
\]
\end{prop}

\begin{proof}
The scheme of the proof is the same as for Theorem \ref{bound:teo}.
In particular we use the notations introduced there.

We choose points $p_1,\ldots p_{k_0}\in D\cap I^-(\partial_-K)$
and numbers $\epsilon_1,\ldots,\epsilon_k$ such that
\[
    (B_R\times\mathbb R)\cap K\subset\bigcup_1^{k_0} I^+_{\epsilon_k}(p_k)\,.
\]

On $I^+_{\epsilon_k}(p_k)$ we consider the time function
$\tau_k=\tau_{p_k}-\epsilon_k$
where $\tau_{p_k}$ denote the Lorentzian distance from 
$p_k$ and is a time function on $I^+(p_k)$.
Notice that $\tau_k$ is smooth on  the domain 
$\mathcal V=I^+(p_k)\cap I^-((p_k)_+)$.

Moreover $\overline{K\cap I^+_{\epsilon_k/2}(p_k)}$ 
is a compact domain in $\mathcal V$.

Since $M^r_s$ is contained in $K$ for every $r$ 
and $s$,  we deduce that there exists $r_0$ such that
for $r\geq r_0$ and $k=1,\ldots,k_0$ 
\[
   \partial M^r_s\cap  I^+(p_k)=\emptyset
\]
and $M_r\cap\{\tau_k\geq 0\}=M^r_s\cap \overline{I^+_{\epsilon_k/2}(p_k)}$ is compact.

Thus we are in the hypothesis of Theorem 2.1 of \cite{ecker}, 
 there is a constant $A_k$ 
\begin{equation}
   \sup_{M^r_s\cap I^+_{\epsilon_k}(p_k)}v_{M^r_s}\leq A_k(1+\frac{1}{s})\,.
\end{equation}      
where $A_k$ depends on 
 the $C^2$ norm of $\tau_k$ and $t$  and the $C^0$ norm of $Ric$
taken on the domain $\overline{K\cap I^+_{\epsilon_k/2}(p_k)}$
with respect to a reference Riemannian metric.

In particular for $s>\eta$ we have
\begin{equation}\label{ecker:eq}
   \sup_{M^r_s\cap I^+_{\epsilon_k}}v_{M^r_s}\leq A_k(1+\frac{1}{\eta})\,.
\end{equation}  

By Theorem 2.2 of \cite{ecker} we also have that for every $m=0,1,\ldots$ there are constants
$A_{k,m}$ such that
\[
   \sup_{M^R_s\cap I^+_{\epsilon_k}(p_k)}|\nabla^m A|^2\leq A_{k,m}\,.
\]

In particular, the constants $C=\sup\{A_1,\ldots A_{k_0}\}$, $C_m=\sup\{A_{1,m},\ldots, A_{k,m}\}$
satisfy the statement.


\end{proof}

\begin{theorem}\label{mcfmain:teo}
There is a family of spacelike functions
\[
  \bar u_s:\mathbb H^n\rightarrow\mathbb R
\]
for $s\in (0,+\infty)$ that verifies (\ref{mcf:eq2}) such that
\begin{itemize}
\item $\bar u_s\rightarrow u$ as $s\rightarrow 0$ in the compact open topology.
\item $\{\bar u_s\}_{s>1}$ is a relatively compact family in $C^\infty(\mathbb H^n)$.
\item the graph $M_s$ of  $\bar u_s$ is contained  in  $K$ for every $s>0$.
\item the mean curvature of $M_s$ satisfies $H_s(x)^2<\frac{n}{2s}$.
\end{itemize}
\end{theorem}

\begin{proof}
For any $R>0$ and $\epsilon>0$
we consider the restriction of $u^r$ on $B_R\times [-\epsilon, +\infty)$.
Proposition \ref{mainmcf:prop} implies that such restrictions form
 a pre-compact  family in $C^\infty(B_R\times[-\epsilon,+\infty))$.
 
 By a diagonal process, we can construct a sequence $r_n\rightarrow+\infty$ such that
 $(u^{r_n})$ converges to $\bar u$ in the $C^\infty$-topology on compact subsets of 
 $\mathbb H^n\times (0,+\infty)$.
 Notice that by construction $(\bar u_s)_{s>1}$ is precompact in $C^\infty(\mathbb H^n)$.
 
 By the uniform estimate on the gradient function of $u^r_s$ on $B_R$ we get that
 the graph $M_s$ of $\bar u_s$ is spacelike.
 Clearly $\bar u_s$ verifies equation (\ref{mcf:eq2}).
 
 Since (\ref{mcest:eq}) holds for every $u^r_s$,
we get that $H(\bar u_s)^2<\frac{n}{2s}$.
 
 Analogously, passing to the limit in the inclusion $M^r_s\subset K$, we get that $M_s$ is contained
 in $K$.
 
 Comparing (\ref{mcf:eq2}) with (\ref{mcest:eq}), it results that
 \[
     |u^r_s(x)-u(x)|\leq \sqrt{ns}\,.
 \]  
 Taking the limit for $r\rightarrow+\infty$ we get
 \[
   |\bar u_s(x)-u(x)|\leq\sqrt{ns}
\]
that  shows that $\bar u_s\rightarrow u$ in the compact open topology.
\end{proof}
 
\begin{remark}
Taking the limit of $M_{s_k}$ for a suitable sequence $s_k\rightarrow+\infty$ we obtain a maximal surface
contained in $D$. Thus Theorem \ref{mcfmain:teo} furnishes another
proof of Theorem \ref{maximal:teo}.
\end{remark}